\documentclass[11pt]{amsart}
\usepackage[parfill]{parskip}    
\usepackage[pdftex]{graphicx,color}
\usepackage{amssymb,amsmath,amsthm}
\usepackage{epstopdf}
\usepackage{wrapfig}
\usepackage{marvosym}
\usepackage{tikz}
\usepackage{verbatim}

\numberwithin{equation}{section}

\usepackage[pdftex,bookmarks,colorlinks,breaklinks]{hyperref}
\definecolor{dullmagenta}{rgb}{0.4,0,0.4}   
\definecolor{darkblue}{rgb}{0,0,0.4}
\hypersetup{linkcolor=red,citecolor=blue,filecolor=dullmagenta,urlcolor=darkblue}

\def\dblue#1{\textcolor[rgb]{0,0,0.7}{#1}}

\newcommand{\ii}{{\rm i}}
\newcommand{\ee}{{\rm e}}
\newcommand{\dd}{{\rm d}}

\newcommand{\CC}{{\mathbb C}}
\newcommand{\RR}{{\mathbb R}}

\newcommand{\DD}{{\mathbb D}}

\def\conn{\mathop{\rm conn}}
\def\clos{\mathop{\rm clos}}
\def\inte{\mathop{\rm int}}

\newtheorem{thm}{Theorem}[section] 

\newtheorem{lemma}[thm]{Lemma}
\newtheorem{cor}[thm]{Corollary}
\newtheorem*{thm-others}{Theorem}

\newtheorem{innercustomthm}{Theorem}
\newenvironment{customthm}[1]
  {\renewcommand\theinnercustomthm{#1}\innercustomthm}
  {\endinnercustomthm}

\theoremstyle{remark}

\def\C{{\mathbf{C}}}
\def\R{{\mathbb R}}

\renewcommand{\descriptionlabel}[1]%
         {\dblue{#1:}\\}

\title{Topology of quadrature domains}
\author{Seung-Yeop~Lee}
\address{Department of Mathematics, California Institute of Technology,\newline Pasadena, CA 91125, USA}
\email{duxlee@caltech.edu}
\thanks{The first author was supported by Sherman Fairchild Senior Research Fellowship.\newline \indent The second author was supported by NSF grant no. 1101735.}
\author{Nikolai~G.~Makarov}
\address{Department of Mathematics, California Institute of Technology,\newline Pasadena, CA 91125, USA}
\email{makarov@caltech.edu}

\begin{document}

\maketitle

The {\it mean value property} for analytic functions states that  if $\Omega$ is a disc,
$$\Omega=B(a,\rho):=\{|z-a|<\rho\},$$then for all bounded analytic functions $f$ in $\Omega$ we have
\begin{equation}\label{01}
\int_\Omega f\,\dd A=\pi \rho^2 f(a),
\end{equation}
where $\dd A$ is the area measure. Similar property holds for cardioids: if
for example
\begin{equation*}
\Omega=\{z+z^2/2:~ |z|<1\},
\end{equation*}
then
\begin{equation}\label{02}
\int_\Omega f \,\dd A= \frac{3\pi}{2} f(0)+\frac{\pi}{2} f'(0).
\end{equation}
 Formulas like \eqref{01} and \eqref{02} are called {\it quadrature identities}, and the corresponding domains of integration are called (classical) {\it quadrature domains}.
 Various classes of quadrature domains have been known for quite some time, see e.g. Neumann's papers \cite{Neumann-oval1,Neumann-oval2} from the beginning of the last century, but the systematic study began only with the work of  Davis \cite{Davis}, and  Aharonov and  Shapiro \cite{AS76}.   We refer  to the monographs \cite{Varchenko-Etingof,Gustafsson-Vasiliev, Davis, Sakai-book-QD, Shapiro92} related to this subject, and
to the survey paper \cite{Gus05} for a quick overview and further references.

 Quadrature domains appear in  several diverse areas of analysis such as  extremal problems for certain classes of analytic functions \cite{AS1,AS2,AS3,AS4}, the Schwarz reflection principle \cite{Davis}, univalent functions, zeros of harmonic polynomials \cite{Duren, poly-book}, and even operator theory  (hypernormal and subnormal operators) \cite{Putinar1,Putinar2}.  An important discovery by Richardson \cite{Richardson72} linked the theory of quadrature domains to the {\it Hele-Shaw flow} in fluid mechanics.

From the point of view of potential theory (the importance of potential theoretic approach was first advocated by Sakai \cite{Sakai1,Sakai2,Sakai3,Sakai-book-QD}), quadrature domains are related to the theory of partial balayage \cite{GusA,GusB,GusC,SakaiA,Sakai-book-QD} and free boundary (or obstacle) problems \cite{free3,free2,free1}, as well as to Richardson's moment problem \cite{Sakai78,GusB, GusInverse}. Also,
as we will see, there is a close connection to the  logarithmic potential theory with an  external field \cite{ST} in the case where the field has an algebraic potential. A basic question  -- how to describe the {\it droplets} (the support of equilibrium measures) for all possible localizations of a given algebraic (e.g. cubic) potential --  was our initial motivation for this research.  This question naturally appears in the context of the random normal matrix theory \cite{KKMWZ,ABTWZ,WZ3}, and it is essentially equivalent to the well-known {\it inverse problem} of (logarithmic) potential theory \cite{Inverse-book}. 

A crucial part of the inverse problem is the problem of {\em topology} of quadrature domains and, more generally, of algebraic droplets. Once we know that the quadrature domain is, say, simply connected, then we can use the Riemann map technique to find a system of algebraic equations describing the boundary. Similarly, in the doubly connected case we would use the annulus uniformization and arrive to a system of equations involving elliptic functions, etc. see \cite{Crowdy,Crowdy_constructingmultiply} for higher connectivity
cases. Another source of motivation to study topology of quadrature domains comes from the problem of laminarity and topological transitions in the Hele-Shaw flows with various applications to mathematical physics and fluid dynamics \cite{Crowdy,Klein}. Finally, topological properties of algebraic droplets is an interesting topic in itself. The boundary of an algebraic droplet is a real algebraic curve (up to a finite set), and the question of possible topological configurations of its components is a traditional topic in real algebraic geometry.

In the present  paper we will address the problem of topology of quadrature domains, namely we will establish {\it upper bounds} on  the {\it connectivity} of the domain in terms of the number of nodes and their multiplicities in the quadrature identity. First results in this direction were obtained by Gustafsson \cite{Gu1} who proved that bounded quadrature domains of order 2 are simply connected but could be multiply connected for higher orders. We will also discuss several applications of the connectivity  bounds to some of the topics mentioned above. 
The connectivity bounds of this paper are in fact {\it sharp}, which is the subject of the companion paper \cite{paper-sharp}.


Our argument is the combination of three techniques: the description of quadrature domains in terms of the potential theory with an algebraic external field, the conformal dynamics of the Schwarz reflection, and the perturbation technique which is based on the Hele-Shaw flow. We should mention that the idea to use methods of complex dynamics originated in Khavinson and \'Swi\c atek's note \cite{Kha1}.

The paper is organized as follows. In the first two sections we introduce the main concepts and state the results of the paper. In particular, in Section 1
  we   state the connectivity bounds for bounded and unbounded quadrature domains as Theorem \hyperref[thmA]{A}, and in Section 2 we
state some consequences of Theorem \hyperref[thmA]{A}. The proofs are presented in the remaining three sections of the paper. In
Section 3 we clarify the relation between quadrature domains and potential theory with an algebraic external field.
In
Section 4 we use the dynamics of the Schwarz reflection to prove  connectivity bounds in the case of  non-singular domains. Finally, in Section 5 we apply methods of the Hele-Shaw flow to deal with singular points and consequently  finish  the proof of Theorem \hyperref[thmA]{A} and related statements.

The authors would like to thank Dmitry Khavinson, Curtis McMullen, and Paul Wiegmann for their interest and useful discussions.

\section{Quadrature domains}

We will first recall the definition of a quadrature domain and then we will state Theorem \hyperref[thmA]{A}, the main result of the paper.

\subsection{Bounded domains} By definition,
a  bounded  connected open set  $\Omega\subset \CC$ is a bounded quadrature domain (BQD) if it carries a finite node {\it quadrature identity}, which means there exists a finite collection of triples $(a_k,m_k,c_k)$, where  $a_k$'s are points (not necessarily distinct) in $\Omega$,  $m_k$'s are non-negative integers, and $c_k$'s are some complex numbers, such that
\begin{equation}\label{QI}
\forall f\in C_{A}(\Omega),\qquad \int_\Omega f\,\dd A = \sum_k c_k f^{(m_k)}(a_k).
\end{equation}
Here,  $C_{A}(\Omega)$ denotes the space of analytic  functions in $\Omega$ which are continuous up to the boundary.
 We will always assume
\begin{equation*}
\Omega=\inte\clos\Omega,
\end{equation*}
for otherwise, it would be trivial to construct infinitely many domains with the same quadrature data (e.g.,  by deleting  subsets of zero area from $\Omega$).

We can rewrite the definition \eqref{QI} by using the contour integral in the right hand side of the quadrature identity:
\begin{equation}\label{QI1}\forall f\in C_A(\Omega), \qquad
\int_\Omega f\,\dd A=\frac{1}{2\ii}\oint_{\partial\Omega} f(z)\,r(z)\,\dd z,
\end{equation}
where
\begin{equation*}
 r(z)\equiv r_\Omega(z)
=\frac{1}{\pi}\sum_{k} c_k\frac{m_k!}{(z-a_k)^{m_k+1}}.
\end{equation*}
The contour integral is understood in terms of residue calculus (in fact, the integral exists in the usual sense because $\partial\Omega$ is an algebraic curve, see \cite{Gu1}), and we always consider $\partial\Omega$ with the standard orientation relative  to $\Omega$.

By taking $f(z)=(z-w)^{-1}$ in \eqref{QI1}, we see that $r$ is {\em uniquely} determined by the quadrature domain as long as we require that  all poles of $r$ be inside $\Omega$ and $r(\infty)=0$.  We will call $r_\Omega$ the {\em quadrature function} and
\begin{equation*}
d_\Omega:=\deg r_\Omega
\end{equation*}
the {\it order} of  the quadrature domain $\Omega$.
The poles of $r_\Omega$ are called the {\em nodes} of $\Omega$.

\subsection{Unbounded domains}

Slightly modifying the statement in \eqref{QI1}, we extend the concept  of  quadrature identities to the case of unbounded domains. Let  $\Omega\subset \widehat\CC$ (the Riemann sphere) be an open connected set such that   $\infty\in\Omega$ and $\inte\clos\Omega=\Omega$. By definition,
$\Omega$ is an unbounded quadrature domain (UQD) if
there exists a rational function $r=r_\Omega$ with no poles outside $\Omega$ such that
\begin{equation}\label{QI2}
 f\in C_A(\Omega),\quad f(\infty)=0 \quad \Longrightarrow  \quad \int_\Omega f\,\dd A =\frac{1}{2\ii}\oint_{\partial\Omega} f(z)\,r(z)\,\dd z.
\end{equation}
The integrals over unbounded domains are always understood in the the sense of  {\em principal value}:
$$\int_\Omega~\equiv~{\rm v.p.}\int_\Omega~:=~\lim_{R\to\infty}~ \int_{\Omega\cap B(0,R)}.
$$

For example,
an exterior disk $$\Omega=B^\text{ext}(a,\rho):= (\clos B(a,\rho))^c$$ is an unbounded quadrature domain of order 0.  Indeed, if $f(\infty)=0$, then
   for all $R>\rho+|a|$ we have
\begin{equation*}
\int_{\Omega\cap B(0,R)} f\,\dd A=\frac{1}{2\ii}\oint_{\partial\Omega} f(z)\, \overline z \,\dd z=\frac{1}{2\ii}\oint_{\partial\Omega} f(z)\, \left(\overline a+\frac{\rho^2}{z-a}\right) \,\dd z=\frac{1}{2\ii}\oint_{\partial\Omega} f(z)\, \overline a \,\dd z,
\end{equation*}
and so
\begin{equation}\label{extcircle}
r_\Omega(z)\equiv \overline a.
\end{equation}

\subsection{First examples}
It is known that disks are the only BQDs of order one, and exterior disks the only UQDs of order 0, see \cite{Epstein1,Epstein2}.
In Figure \ref{fig-1} we show some examples of BQDs of order two and UQDs of order one.

\begin{figure}[ht]
\begin{center}
\includegraphics[width=0.7\textwidth]{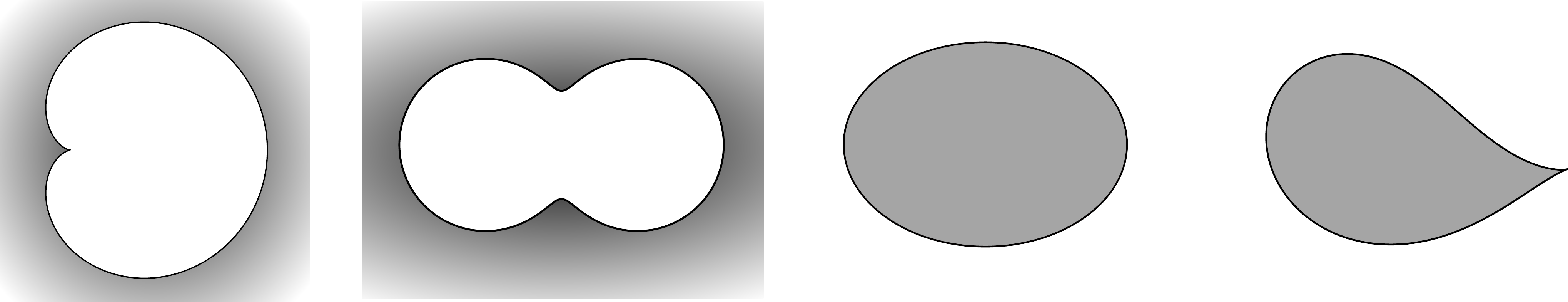}
\end{center}
\caption{\label{fig-1} From left to right: cardioid; Neumann's oval; ellipse; Joukowsky's airfoil.  The shaded parts are complements of quadrature domains.}
\end{figure}
\begin{itemize}
\item There are two types of BQDs of order two -- with a single node and with two nodes. The domains of the first type are
{\it lima\c{c}ons}; in the special case where the boundary has a cusp, the domain is a {\it cardioid}.
Lima\c{c}ons as examples of quadrature domains were identified by Polubarinova-Kochina \cite{PK1,PK2} and Galin \cite{Galin} in the context (and language) of the Hele-Shaw problem. {\it Neumann's ovals} \cite{Neumann-oval1,Neumann-oval2} are BQDs such that $r_\Omega$ has two simple poles with equal residues. The boundary of a Neumann oval can be  obtained by reflecting an ellipse in a concentric circle.

\item
Unbounded quadrature domains of order 1  also come in two varieties -- depending on the location of the node.  If the node is at $\infty$, then the domain is the exterior of an ellipse. Examples with a finite node include the exteriors of
Joukowsky's {\it airfoils}. The boundary of an airfoil is a Jordan curve with a cusp; the curve is the image of a circle under  Joukowsky's map $z\mapsto z+\frac1z$.

\end{itemize}

\bigskip
{\bf Remarks.}
\begin{itemize}
\item[(a)] {\it Circular inversion}

It is not difficult to show that the reflection in the unit circle provides a one-to-one correspondence between the class of bounded quadrature domain $\Omega$ of order $d+1$ satisfying $0\in\Omega$ and the class of unbounded quadrature domain of order $d$ satisfying $0\notin\clos\Omega$.    This follows, for instance, from the Schwarz function characterization of quadrature domains; see Remark (b) in Section \ref{sec-Sch-refl}.

On the other hand, there is no simple way to relate the quadrature data (multiplicities and location of the nodes) under the circular inversion.  For example, a circular inversion of the exterior of an airfoil can have one or two distinct nodes.
This is the reason why we often need to consider the cases of  bounded and unbounded quadrature domains separately.

\item[(b)] {\it Univalent rational functions}

All examples in Figure \ref{fig-1} involve simply-connected quadrature domains.  A simply-connected domain is a quadrature domain if and only if the corresponding Riemann map is a rational function.    The theory of univalent functions, see e.g. \cite{Duren}, provides many explicit examples of simply connected quadrature domains of higher order.

\end{itemize}

\subsection{Connectivity bounds}

Applying methods of Riemann surface theory, Gustafsson \cite{Gu1} showed that all BQDs of order 2 (and therefore all UQDs of order 1) are simply-connected, so examples in Figure \ref{fig-1} represent exactly all possible cases.  At the same time, referring to Sakai's work \cite{Sakai-book-QD,Sakai83}, Gustafsson proved the existence of BQDs of connectivity $2d-4$ for all $d=d_\Omega\geq 3$.  (Earlier, Levin \cite{Levin} constructed bounded, doubly-connected domains that satisfy quadrature identities of order 2 for all analytic functions with single-valued primitives.)

The main goal of this paper is to give upper bounds on the connectivity of quadrature domains in terms of multiplicities of their nodes.  In particular, we will see that Sakai-Gustafsson's examples are best possible if all nodes are simple.
Our results, which we state in Theorems \ref{thm-A1} and \ref{thm-A2} below, have different forms for bounded and unbounded quadrature domains.   As we explain in the next subsection, the inequalities in these theorems are sharp.
For a quadrature domain $\Omega$, we denote
\begin{equation*}\label{calK}
\conn\Omega = \#(\text{components in $\widehat\CC\setminus\Omega$}),
\end{equation*}
and
\begin{equation*}
n_\Omega=\#\text{(distinct poles of $r_\Omega$)}.
\end{equation*}

\phantomsection
\label{thmA}

\bigskip
\begin{customthm}{A1}\label{thm-A1}
If $\Omega$ is an unbounded quadrature domain of order $d_\Omega\geq 2$, then
\begin{equation}\label{eq-thma1}
\conn\Omega~\leq~ \min(d_\Omega+n_{\Omega}-1,~ 2d_\Omega-2).
\end{equation}
If, in addition, one of the  nodes is at $\infty$, then
\begin{equation}\label{eq-thma2}
\conn\Omega~\leq~ d_\Omega+n_{\Omega}-2.
\end{equation}
\end{customthm}
\bigskip
\begin{customthm}{A2}\label{thm-A2}
If $\Omega$ is a bounded quadrature domain of order  $d_\Omega\ge3$, then
\begin{equation}\label{eq-thma3}
\conn\Omega~\leq~ \min(d_\Omega+n_{\Omega}-2,~2d_\Omega-4).
\end{equation}
If, in addition, there are no nodes of  multiplicity $\ge3$, then
\begin{equation}\label{eq-thma4}
\conn\Omega~\leq~ \min(d_\Omega+n_{\Omega}-3, ~2d_\Omega-4).
\end{equation}
\end{customthm}
We will refer to these two theorems collectively as Theorem \hyperref[thmA]{A}.
As we mentioned, if $d_\Omega=1$ for an unbounded $\Omega$ or $d_\Omega\leq 2$ for a bounded $\Omega$, then  the quadrature domain is simply connected.

Let us also emphasize the special case when $\Omega$ has a single node ($n_\Omega=1$).
\bigskip
\begin{cor}\label{cor1} If $\Omega$ is a UQD such that $r_\Omega$ is a polynomial, or if $\Omega$ is a BQD with a single node, then
\begin{equation*}
\conn\Omega\leq d_\Omega-1,\qquad (d_\Omega\ge2).
\end{equation*}
\end{cor}

\subsection{Sharpness of connectivity bounds}\label{sec-sharp}

The bounds in Theorem \hyperref[thmA]{A} are best possible in the following sense.
\bigskip

\begin{customthm}{B1}\label{thm-B1}
Given numbers $d\geq 2$, $n\leq d$ and a partition $d=m_1+\cdots+m_{n}$ (where $m_j$'s are positive integers),
there exists a UQD $\Omega$ of order $d$ with $n$ finite nodes of multiplicities $m_1,\cdots,m_{n}$ such that
\begin{equation*}
\conn\Omega = \min(d+n-1,~ 2d-2) .
\end{equation*}
Also, there exists a UQD of order $d$ with a node of multiplicity $m_n$ at $\infty$ and with $n-1$ finite nodes of multiplicities $m_1,\cdots,m_{n-1}$ such that
\begin{equation*}
\conn\Omega = d+n-2.
\end{equation*}
\end{customthm}
\bigskip
\begin{customthm}{B2}\label{thm-B2}
Given numbers $d\geq 3$, $n\leq d$, and a partition $d=m_1+\cdots+m_{n}$, there exists a BQD $\Omega$ of order $d$ with node multiplicities $m_1,\cdots,m_{n}$ such that
\begin{equation*}
\conn\Omega = \min(d+n-3,~ 2d-4).
\end{equation*}
If at least one of $m_j$ is $\geq 3$, then there exists a BQD with node multiplicities $m_1,\cdots,m_{n}$ such that
\begin{equation*}
\conn\Omega = d+n-2.
\end{equation*}
\end{customthm}
For example, there are four possible cases for unbounded quadrature domains of order 2:
\begin{itemize}
\vspace{-0.4cm}
\item[(i)] $n=1$, the pole is finite;
\vspace{-0.4cm}
\item[(ii)] $n=2$, both poles are finite;
\vspace{-0.4cm}
\item[(iii)] $n=1$, the pole is infinite;
\vspace{-0.4cm}
\item[(iv)] $n=2$, one pole is finite, the other is $\infty$.
\end{itemize}
In the first two cases, according to Theorem \ref{thm-B1} there are UQDs $\Omega$ such that
\begin{equation*}
\conn\Omega=\min(d+n-1,~2d-2)=2.
\end{equation*}
In cases (iii) and (iv),
\begin{equation*}
\conn\Omega=d+n-2=\begin{cases} 1 \quad \text{ case (iii)}, \\2 \quad \text{ case (iv)}.\end{cases}
\end{equation*}
\begin{figure}[ht]
\begin{center}
\includegraphics[width=0.9\textwidth]{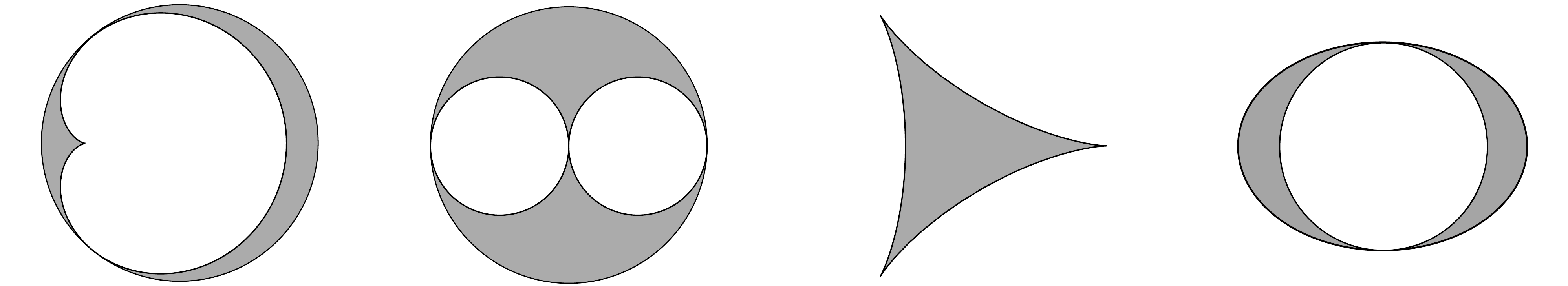}
\caption{\label{fig-sharpUQD} Construction of UQD of order $2$}
\end{center}
\end{figure}

The pictures in Figure \ref{fig-sharpUQD} illustrate (and basically prove) the existence of such quadrature domains.  The unshaded regions (e.g. the cardioid and the exterior disc in the first picture) are the unions of disjoint quadrature domains.  The sum of the orders of quadrature domains is 2 in each picture, and multiplicities and positions of the nodes correspond to our cases (i)-(iv).  By a small perturbation that preserves the quadrature data (the sum of quadrature functions) we can transform each union of quadrature domains into a single connected quadrature domain.  This way we obtain examples of quadrature domains of maximal connectivity.  The perturbation procedure will be explained in Section \ref{sec-HSflow}.

\begin{figure}[ht]
\begin{center}
\includegraphics[width=0.7\textwidth]{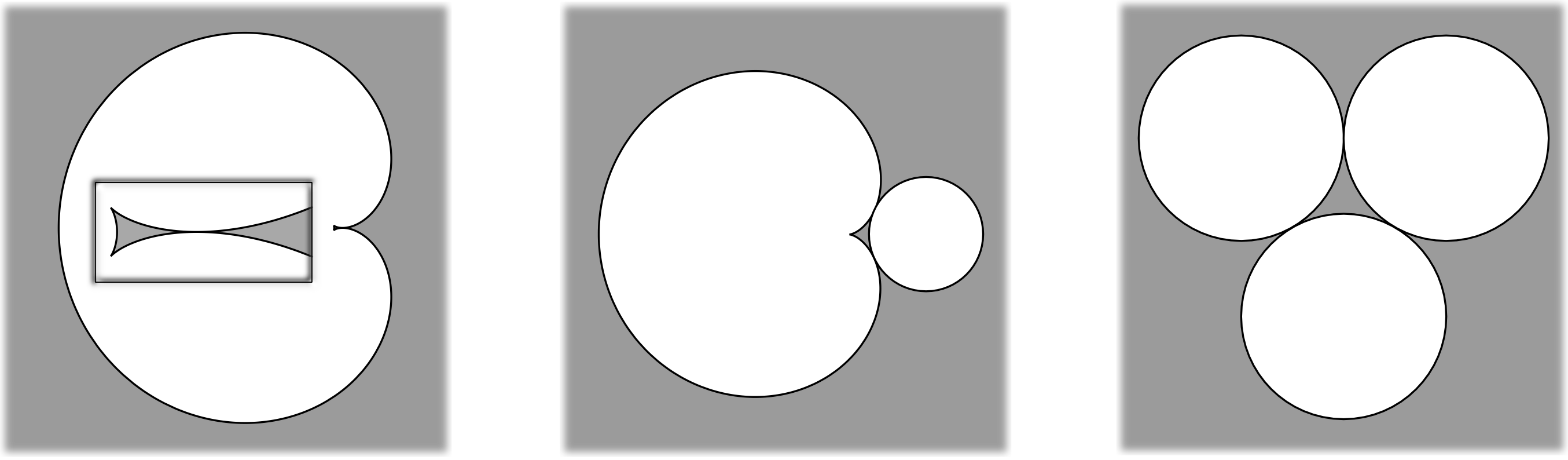}
\caption{\label{fig-sharpBQD3} Construction of BQD of order $3$}
\end{center}
\end{figure}

Similarly, there are three possibilities for BQDs of order $d=3$: they correspond to the partitions
\begin{equation*}
3=3, ~~3=2+1,~~ 3=1+1+1.
\end{equation*}
In the first case, $\Omega$ has a triple node (i.e. $n=1$) and \eqref{eq-thma3} implies
\begin{equation*}
\conn\Omega\leq \min(3+3-2,6-4)=2.
\end{equation*}
In the cases  $3=2+1$ and $3=1+1+1$ we apply \eqref{eq-thma4}, and also get
\begin{equation*}
\conn\Omega\leq 2.
\end{equation*}
The existence of doubly connected quadrature domains in all three cases (Theorem \ref{thm-B2}) follows from Figure \ref{fig-sharpBQD3}.  The leftmost picture (the case of a triple node) is the image of the unit disc under the univalent polynomial
\begin{equation}\label{poly3}
z\mapsto z -\frac{2\sqrt 2}{3} z^2 + \frac{1}{3} z^3.
\end{equation}
The boxed inset gives a magnified view near the cusp of the exterior boundary.  The polynomial \eqref{poly3} was discovered by Cowling and Royster \cite{Cowling-Royster}, and Brannan \cite{Brannan}.

In the companion paper \cite{paper-sharp} we extend the above construction to prove Theorem \ref{thm-B1} and \ref{thm-B2} for all values of $d$.  The main fact is the existence of univalent rational functions similar to \eqref{poly3}, which can be used to show the sharpness of the bound in Corollary \ref{cor1}.  We also explain in \cite{paper-sharp} that the sharpness results for UQDs in the cases $n=d$ and $n=1$ are closely related to the sharpness results by Rhie \cite{Rh03} and, respectively, by Geyer \cite{Ge03} concerning the maximal number of solutions to the equation $\overline z= r(z)$ where $r$ is a rational function.

\section{Algebraic  droplets}

In this section we discuss some application of the connectivity bounds in Theorem \hyperref[thmA]{A} to logarithmic potential theory with an external field.   More specifically we consider the case where
the external field has an {\it algebraic Hele-Shaw } potential.  The definition of algebraic Hele-Shaw potentials is given below in Section \ref{sec-AHS} and the relation to the Hele-Shaw flow is explained in Section \ref{sec-HSflow}.
The problem of topological classification of all possible shapes of the support of the equilibrium measure first appeared in the context of random normal matrix models, see Section \ref{sec-RMT}.

\subsection{Logarithmic potential theory with external field}\label{sec-log}
Given a function (called the {\it ``external potential''})
$$Q:\CC\to\RR\cup\{+\infty\}$$
we define, for each positive Borel measure $\mu$ with a compact support in $\CC$,  the weighted logarithmic $Q$-energy $I_Q[\mu]$ by the formula
\begin{equation}\label{L} I_Q[\mu]=\int_{\CC\times\CC} \log\frac{1}{|z-w|}\dd\mu(z)\,\dd\mu(w)+ \int_\CC Q\,\dd\mu.
\end{equation}
Physical interpretation: $\mu$ is an electric charge distribution and $I_Q[\mu]$ is the {\em total electrostatic energy} of $\mu$, the sum of the 2D Coulomb energy and the energy of interaction with  the external field.

We always assume that $Q$ is {\it lower  semi-continuous} (in particular the  expression for $I_Q[\mu]$ makes sense), and that $Q$ is finite on some set of positive logarithmic capacity.  A typical situation that we will encounter is when $Q$ is finite and continuous on some closed set with non-empty interior and $Q=+\infty$ elsewhere.  Under these conditions, the classical Frostman's theorem (see \cite{ST}) states that for each $t>0$ such that
\begin{equation}\label{growth}
Q(z)-t\log|z|^2\to+\infty,\qquad z\to\infty,
\end{equation}
there exists a unique ({\it equilibrium}) measure $\sigma_t$ of mass $t$ that minimizes the $Q$-energy:
\begin{equation}\label{II}
I_Q[\sigma_{t}]=\min_{\|\mu\|=t} I_Q[\mu].
\end{equation}
Let us denote
\begin{equation}\label{St}
S_t\equiv S_t[Q]={\rm supp}(\sigma_t).
\end{equation}
It can be shown (see  \cite{HM2011}) that if $Q$ is {\it smooth in some neighborhood  of} $S_t$ (and satisfies the growth condition \eqref{growth} at infinity) then the equilibrium measure is absolutely continuous, and in fact it is given by the formula
\begin{equation*} \dd \sigma_{t} = \textstyle \frac{1}{4\pi}\Delta Q\cdot {\bf 1}_{{\mathcal S}_t}\,\dd A,\end{equation*}
where $\Delta Q$ is the  Laplacian and ${\bf 1}_{{\mathcal S}_t}$ is the indicator function.
In this case, we  refer to $S_t$ as the {\it droplet} of $Q$ of mass $t$.  The point is that we can recover the equilibrium measure from the shape of the droplet.

It is not easy (if at all possible) to find the shapes of the droplets for general external potentials but  there are interesting  explicit examples, see e.g. \cite{ABTWZ}, in the ``algebraic" case that we describe next.

\subsection{Algebraic Hele-Shaw potentials. Local droplets}\label{sec-AHS} The  class of external potentials that  we will consider consists of {\em localized algebraic Hele-Shaw potentials}. Let us explain the terminology.

 A smooth function $Q:{\mathcal O}\to{\mathbb R}$, where ${\mathcal O}$ is some open set in  ${\CC}$, is called a {\it Hele-Shaw potential} if $\Delta Q$ is constant in ${\mathcal O}$. We can choose this constant equal to $4$, so  Hele-Shaw potentials are functions of the form
\begin{equation*}
Q(z)=|z|^2 - H(z),\qquad H \;\text{is harmonic in}~ {\mathcal O}.
\end{equation*}
A Hele-Shaw potential is called an {\it algebraic potential} if
$$h:=\partial H \;\text{is a  rational function};$$
where $\partial$ means the complex derivative $\partial/\partial z$. We will call  $h$ the {\it quadrature function}
 of the algebraic potential $Q$.

We want to emphasize that {\em algebraic} potentials, considered as function on the full plane $\CC$, do not satisfy the conditions of the Frostman theorem and therefore cannot be used as external potentials in the variational problem of minimizing $Q$-energy.  The only exception is the case
$$Q(z)=|z|^2-\Re(az^2+2bz)-\sum c_j\log|z-z_j|^2,\qquad |a|<1,\; c_j>0,$$
where we can extend $Q$ to a continuous map $\CC\to\RR\cup\{+\infty\}$ which satisfies the growth condition \eqref{growth} for all $t>0$. For example, if
$Q(z)=|z|^2-\Re(az^2+2bz)$ and if $|a|<1$, then the droplets are concentric ellipses. On the other hand, if $|a|\ge 1$ or if the quadrature function $h$ is a non-linear polynomial, e.g. $h(z)=z^2$, then  the variational problem \eqref{II} has no solution.

This leads us to the concept of local droplets \cite{HM2011}. By definition, a compact set $K\subset{\mathcal O}$ is a {\it local droplet} of $Q$ if the  measure $\textstyle\frac{1}{4\pi}\Delta Q\cdot {\bf 1}_{K}\,\dd A$, which is just the normalized area measure of $K$ in the case of Hele-Shaw potentials, is the equilibrium measure of mass $t=A(K)/\pi$ of a {\it localized  potential}
\begin{equation}\label{local}
Q_{L} := Q\cdot {\bf 1}_{L} +\infty\cdot{\bf 1}_{\CC\setminus L},
\end{equation}
for some compact set $L\subset {\mathcal O}$.
For instance, $K$ is a local droplet if there is a neighborhood $U\subset\mathcal O$ of $K$ such that  $K$ is a (non-local) droplet of the potential  $Q_{\clos U}$.  We call such local droplets {\it non-maximal}.

The following relation between algebraic droplets and quadrature domains is central for this paper.

Let $K$ be a {\em local droplet of an algebraic potential} $Q$ (an {\em algebraic droplet} for short) with quadrature function $h$.  Then $K^c$ is a union of finitely many quadrature domains, $K^c=\bigsqcup\Omega_j$, and
\begin{equation}\label{HS-QD}
   h = \sum r_{\Omega_j} .
 \end{equation}
The converse is also true: if the complement of a compact set is a disjoint union of quadrature domains then $K$ is a local droplet for some algebraic potential.

These statements will be explained  in Section \ref{sec-dual} (Theorems \ref{thm-HS-QD} and \ref{QDtodroplet}).

\subsection{Random normal matrix model and Richardson's moment problem}\label{sec-RMT}

Recall that the eigenvalues $\{\lambda_j\}\in \CC^n$ of the $n\times n$ matrices in the random normal matrix model with a given potential $Q$ are distributed according to the probability measure
\begin{equation}\label{partition-fct}
\ee^{-H(\lambda)}\dd\lambda~\bigg/~\int\ee^{-H(\lambda)}\dd\lambda\qquad (\lambda\in\CC^n),
\end{equation}
where $\dd\lambda:=\dd A(\lambda_1)\dots \dd A(\lambda_n)$ and
\begin{equation*}
H(\lambda)=\sum_{i\ne j}\log\frac{1}{|\lambda_i-\lambda_j|} + n\sum_i Q(\lambda_i).
\end{equation*}
Comparing the last expression with \eqref{L}, it is natural to expect that the random measures
\begin{equation*}
\mu_n=\frac{1}{n}\sum_j\delta_{\lambda_j}
\end{equation*}
converge to the equilibrium measure of mass one as $n\to\infty$, and according to \cite{Joh}, see also \cite{EF} and \cite{HM2011}, this is indeed the case if $Q$ satisfies the conditions of the Frostman theorem.

More generally, for $t>0$, the eigenvalues of the random normal matrix with potential $Q/t$ condensate on the set $S_t[Q]$, see \eqref{St}.

The random normal matrix model with $Q(z)=|z|^2$ has been extensively studied (the distribution of eigenvalues is known as the complex Ginibre ensemble) as well as its immediate  generalizations $Q(z)=|z|^2-\Re(az^2+bz)$ with $|a|<1$.
One of the first ``non-Gaussian'' examples, the case of the ``{\it cubic}'' potential $Q(z)=|z|^2 - \Re z^3$, was considered in \cite{ABTWZ}.  Despite the fact that the model is not well-defined (the integral in  \eqref{partition-fct} diverges), the authors constructed (somewhat formally) the family of ``droplets'' shown in Figure \ref{fig-3}.  This is an increasing family of closed Jordan domains bounded by certain hypotrochoids.  The boundary of the largest domain has 3 cusps; the curve is known as the {\it deltoid} curve.  Elbau and Felder \cite{EF} suggested a mathematical meaning of formal computations in terms of certain cut-offs (or localizations) of the potential, and in fact one can show that the compact sets  in Figure \ref{fig-3} are local droplets of $Q$, and these droplets  are  non-maximal except for the deltoid, see the text below \eqref{local}.

\begin{figure}[ht]
\begin{center}
\includegraphics[width=0.25\textwidth]{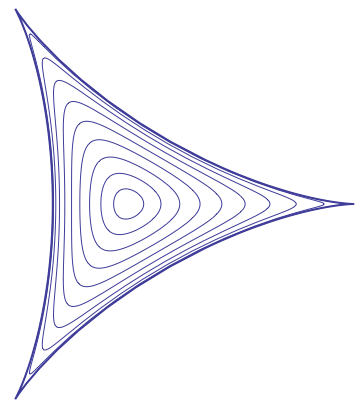}
\caption{\label{fig-3} Local droplets of the cubic potential}
\end{center}
\end{figure}

There are infinitely many ways to localize a given potential, i.e. to choose a compact set $L$ in the definition \eqref{local} of localized potentials $Q_L$ and local droplets of $Q$, so the question arises whether the hypotrochoids in \cite{ABTWZ} represent all possible local droplets of the cubic potential and, in particular, whether the maximal mass $t$ that a local droplet can have is the area of the deltoid.  The answer is ``yes'', and the proof  depends on the fact that  unbounded quadrature domains of order two with a double node at infinity are simply connected (see Corollary \ref{thm-theyare}).

\bigskip
\begin{thm}\label{thm-theyare} If $h$ is a quadratic polynomial, then there is at most one (maybe none) local droplet of a given area such that $h$ is its quadrature function.
\end{thm}

See details in Section \ref{sec-proof-unique}.

This theorem has an interpretation in terms of the {\it inverse moment problem} that we describe below.
For a domain $\Omega\subset\widehat\CC$ such that
\begin{equation}\label{eqOmega}
\infty\in\Omega,\quad 0\not\in\clos\Omega,\quad \Omega=\inte\clos\Omega,
\end{equation}
we define the moments
$$  m_0=A(\Omega^c),\quad m_{k}=\int_\Omega z^{-k}\dd A(z),\quad (k=1,2,\cdots).$$
It is easy to see that the moments don't determine $\Omega$ in the class of multiply-connected domains (e.g., compare the moments of a disk and an annulus).  In fact, we have a similar non-uniqueness phenomenon for simply-connected domains if we don't require that the closures of the domains be simply-connected; see, for example, the construction in \cite{Sakai87}.  Furthermore, there are non-uniqueness examples for Jordan domains with infinitely many non-vanishing moments, see \cite{Sakai78}, but it is a well-known open problem to construct two distinct Jordan domains with equal moments $m_k$ such that $m_k=0$ for $k\geq k_0$.  In this regard, Sakai \cite{Sakai81} proved this is impossible for $k_0=3$.

\bigskip
\begin{cor}\label{thm-IMP} If two unbounded Jordan domains, $ \Omega$ and $\tilde\Omega$, satisfy \eqref{eqOmega} and have the same moments such that $m_k=0$ for all $k\ge4$, then $\Omega=\tilde\Omega$. \end{cor}

\begin{proof} Denote
\begin{equation*}
	r(z)=\frac{1}{\pi}\sum_{k\geq 1} m_k z^{k-1},
\end{equation*}
so $r$ is a quadratic polynomial.  We claim that $\Omega$ is an unbounded quadrature domain with quadrature function $r$ and the same is true for $\widetilde\Omega$.  According to the definition of unbounded quadrature domains \eqref{QI2}, we need to check that
\begin{equation}\label{fAfr}
	\int_\Omega f(z) \,\dd A(z)=\frac{1}{2\ii} \oint_{\partial\Omega} f(z)\,  r(z) \,\dd z
\end{equation}
for all $f\in C_A(\Omega)$ such that $f(\infty)=0$.  Since $\Omega$ is a Jordan domain, it is sufficient to do so for $f(z)=z^{-k}$ with $k\geq 1$. In this case  the left hand side in \eqref{fAfr} is $m_k$ by definition, and the right hand side is $m_k$ by residue calculus.

It follows that $K=\Omega^c$ and $\widetilde K=\widetilde \Omega^c$ are algebraic droplets with the same quadrature function, which is a quadratic polynomial.  Since the droplets have the same area  $m_0$, we have $K=\widetilde K$ by Theorem \ref{thm-theyare}.
\end{proof}

\subsection{Topology of algebraic droplets}

Let $K\subset\CC$ be a compact set such that $K=\clos\inte K$ and $\partial K$ is a finite union of disjoint simple curves. We will call such curves  {\it ovals} and $\partial K$ a {\it configuration} of ovals. Clearly, any collection of disjoint simple curves is a configuration of ovals; the set $K$ is determined by these curves uniquely. Two configurations of ovals are topologically {\it equivalent} if there is a homeomorphism of $(\widehat\CC,\infty)$ that maps the ovals to the ovals. Our goal  is to describe, in the spirit of Hilbert's 16th problem \cite{Wilson}, all possible configurations of ovals that can occur   in the case  of algebraic droplets  of  a given degree. (By definition the degree of a  droplet is the degree of its quadrature function.)

Let us denote by $q=q(K)$ the number of  components of the complement  $\widehat\CC\setminus K$, and by $q_j$  the number of components    of  connectivity $j$, ($j\ge1$). For example,  we have $q_1=2$ and $q_2=1$ if $\partial K$ is the configuration of 4 concentric circles, and $q_1=1$ and $q_2=2$ in the case of 5 concentric circles; in both cases, $q_{\ge3}=0$. Clearly, we have
$$q=\sum q_j,$$
and we will also write
$$q_\text{odd}=q_1+q_3+q_5+\cdots. $$

\bigskip
\begin{thm}\label{thm-B} (i)
Let $K$ be  a local  droplet of an algebraic Hele-Shaw potential $Q(z)=|z|^2-H(z)$ such that the (rational) quadrature function $h=\partial H$ has degree $d$. Assume that the boundary, $\partial K$, is smooth (i.e. $\partial K$ is a configuration of ovals). Then
\begin{equation}\label{ovals}
 \#(\text{ovals})  + q_\text{odd} +4 (q-q_1)\le 2d+2.
\end{equation}
(ii) Given $d\geq 0$ and given a configuration of ovals satisfying \eqref{ovals}, there exists a local droplet $K$ of some algebraic potential of degree $d$ such that
 $\partial K$ is equivalent to the given configuration.
\end{thm}

The proof will be given in Section \ref{sec-proof-thmC}.

{\bf Examples.} The boundary of an algebraic droplet of degree $d$ can be equivalent to the configuration of 4 concentric circles if and only if $d\ge 4$ because in this case we have
$$\#(\text{ovals})  + q_\text{odd} +4 (q-q_1)=4+2+4(3-2)=10.$$
In the case of 5 concentric circles we have
$$\#(\text{ovals})  + q_\text{odd} +4 (q-q_1)=5+1+4(3-1)=14,$$
and such a configuration is possible if and only if $d\ge6$.

In Figure \ref{fig-2} we display a complete list of possible oval configurations corresponding to algebraic droplets of degree $\le4$.

\begin{figure}[ht]
\begin{center}
\includegraphics[width=0.9\textwidth]{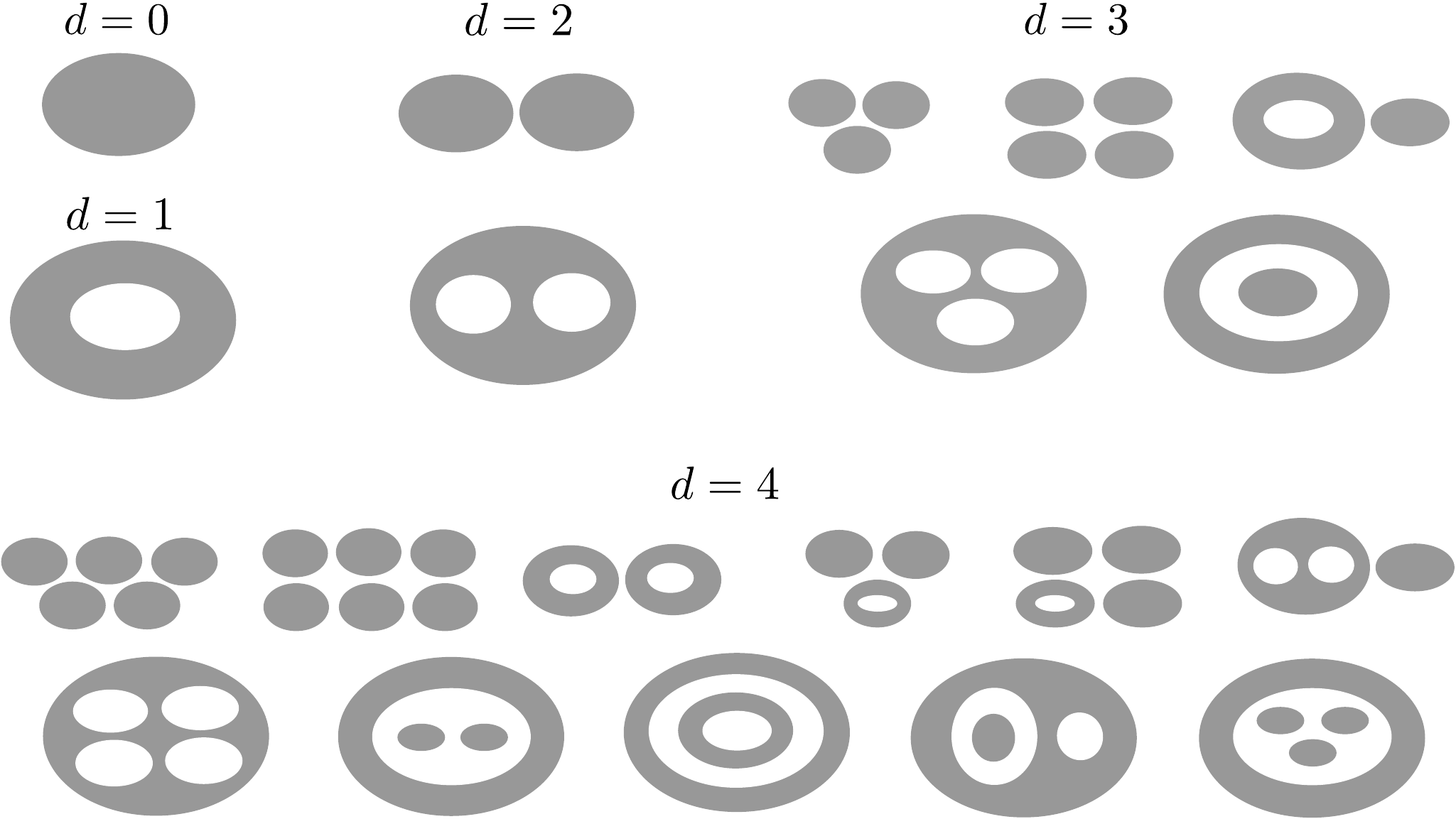}
\caption{\label{fig-2} Topological classification of algebraic droplets according to  the minimal degree  of the quadrature function}
\end{center}
\end{figure}

\bigskip
\begin{cor}\label{cor-R16} Let $K$ be an algebraic  droplet (with smooth boundary)  of degree $d\ge3$.     Then
$$\#(\text{ovals})\leq 2d-2\quad \text{ for } d\geq 3.$$
\end{cor}
\begin{proof}
Indeed, this is immediate from \eqref{ovals} if $q>q_1$.  Otherwise, $q=q_1=q_{\rm odd}$, and  since
 $\widehat\CC\setminus K$ has only simply-connected components, each oval corresponds to a single component of $K^c$, so $\#(\text{ovals})=q$.
We have
\begin{equation*}
2\,\#(\text{ovals})= \#(\text{ovals})+ q=\#(\text{ovals})  + q_\text{odd} +4 (q-q_1)\le2d+2,
\end{equation*}
and $\#(\text{ovals})\le d+1<2d-2$ if $d>3$.   For $d=3$ the estimate follows from the classification in Figure \ref{fig-2}.
\end{proof}

{\bf Remark.}
One can state more detailed results that take multiplicities of the poles of $h$ into account. In particular, if $h$ is a polynomial of degree $d\ge2$, then
the number of ovals is $\le d-1$. This is of course just a reformulation of Corollary \ref{cor1}.

\subsection{Hele-Shaw flow of algebraic droplets}\label{sec-HSflow}

If $K$ is an algebraic droplet (or, more generally, a local droplet of some Hele-Shaw potential) of area $\pi t_0$, then the family
$$K_t:=S_t[Q_K], \qquad (0<t\le t_0),$$
(see \eqref{St} and \eqref{local} for the meaning of $S_t[\,\cdot\,]$ and $Q_K$) is a unique {\em generalized} solution of the {\em Hele-Shaw equation} with sink at infinity:
\begin{equation*}
\frac1\pi\frac d{dt} 1_{K_t}=\omega^\infty_{K_t},\qquad K_{t_0}=K,
\end{equation*}
where $\omega^\infty$ is the harmonic measure evaluated at infinity.  The equation is understood in the sense of integration against test functions.  The family $\{K_t\}$ is called the Hele-Shaw chain of $K$; note that the mass $t$ becomes the ``time'' parameter of the ``flow''.
In the ``classical" case where $\{\partial K_t\}$ is a smooth family of smooth curves, the Hele-Shaw equation means
\begin{equation}\label{Darcy}
V_n(\cdot)=  \nabla G(\,\cdot\,,\infty)\quad \text{on} ~\partial K_t,
\end{equation}
where $V_n$ is the normal velocity of the boundary and $G$ is the Green function with Dirichlet boundary condition on $\partial K_t$.
In 2D hydrodynamics, the classical Hele-Shaw equation \eqref{Darcy} (also known as Darcy's law) describes the motion of the boundary between two immiscible fluids, see \cite{Gustafsson-Vasiliev} for references.

By construction, all droplets in the Hele-Shaw chain of an algebraic droplet have the {\em same quadrature function}.  This simple fact will be useful in perturbation arguments.

{\bf Example.} Suppose we have $m$ disjoint open discs $B_j=B(a_j, \rho_j)$ inside a closed disc $\overline B=\clos B(a, \rho)$ as in the left picture in Figure \ref{fig-packing}. Denote $$K=\overline B\setminus \bigcup B_j$$ and let $c$ be the number of components of the interior of $K$; e.g. $m=9$, $c=16$ in the picture.
Then $K$ is an algebraic droplet of degree $d=m$ with quadrature function
\begin{equation*}
h(z)=\overline a+\sum_{j=1}^m\frac{\rho^2_j}{z-a_j},
\end{equation*}
see \eqref{extcircle} and \eqref{HS-QD}, and this quadrature function does not change under the Hele-Shaw flow $\{K_t\}$, $K_{t_0}=K$.  From basic properties of Hele-Shaw chains (see the review in Section \ref{sec-HSdetail}) it follows that if $t<t_0$ is sufficiently close to $t_0$, then $K_t$ has at least $c$ simply-connected components so $\widehat\CC\setminus K_t$ is an unbounded quadrature domain of connectivity $\geq c$.  Theorem \ref{thm-A1} then gives us the following (sharp) bound for this circle packing problem:
\begin{equation*}
c\leq 2m-2.
\end{equation*}
Of course, there are other ways to obtain this result; e.g. apply Bers' area theorem to the corresponding reflection group.  However, it is less clear whether such alternative arguments extend to the case of more general packing problems like the one depicted in the right picture of Figure \ref{fig-packing}.

\bigskip
\begin{cor} Suppose we have $m$ disjoint open cardioids inside an ellipse, and let $c$ be the number of components in the interior of the complement.  Then $c\leq 3m$.
\end{cor}

E.g. $m=7$ and $c=21$ at the right picture in Figure \ref{fig-packing}.

\begin{proof}
The quadrature function of $K$ has degree $d=1+2m$, and there are $n=m+1$ distinct poles, one of which is at infinity.  Applying \eqref{eq-thma2} in Theorem \ref{thm-A1} to some (small) Hele-Shaw perturbation of $K$ we get $c\leq d+n-2=3m$.
\end{proof}

\begin{figure}[ht]
\begin{center}
\includegraphics[width=0.8\textwidth]{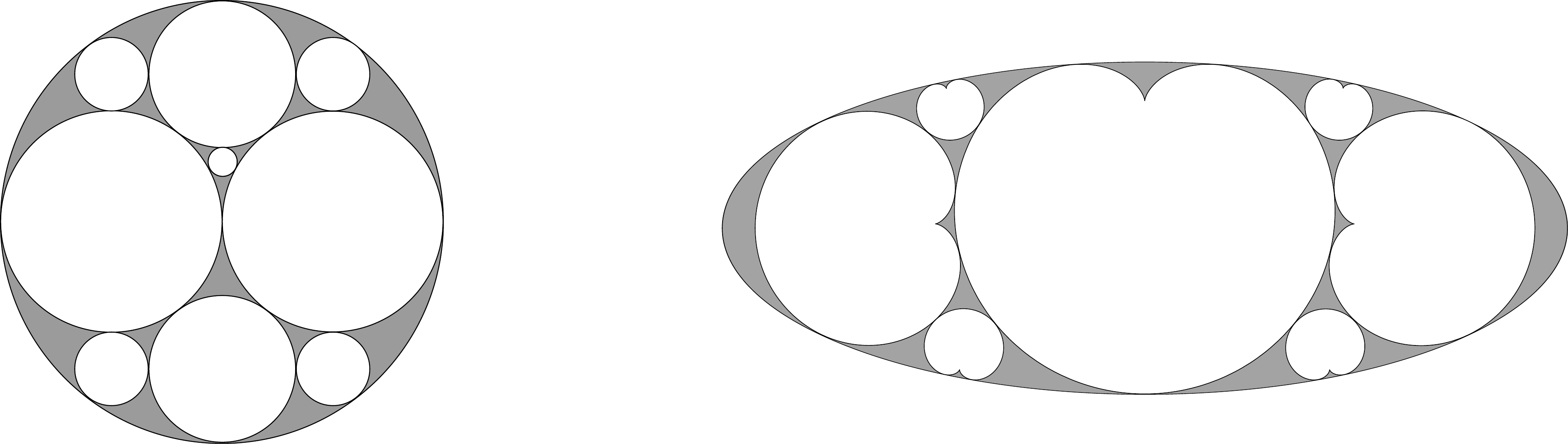}
\caption{\label{fig-packing} Packing problems}
\end{center}
\end{figure}

\section{Quadrature domain decomposition  of the complement of an algebraic droplet}\label{sec-4}

Here we explain the relation between algebraic droplets and quadrature domains.  The argument is based on the Aharonov-Shapiro characterization of quadrature domains in terms of the Schwarz function \cite{AS76}. Since the Schwarz function will be our main tool in the proof of Theorem \hyperref[thmA]{A}, we will also recall some basic facts of its theory.

\subsection{Schwarz function}\label{sec-Schwarz}

Let $\Omega\subsetneqq\widehat\CC$ be a domain such that
$\infty\notin\partial\Omega$ and $\Omega=\inte\clos\Omega.$
By definition, a continuous function
 $$S:\clos\Omega\to\widehat\CC$$ is a  {\em Schwarz function} of $\Omega$ if
 $S$ is meromorphic in $\Omega$ and
 $$S(z)=\overline z\quad{\rm on }~ \partial\Omega.$$
It is clear that, given $\Omega$, if such a function $S$ exists, then
it is unique.  

Our definition deviates from the primary references \cite{Davis,Shapiro92} where Schwarz function is just required to be analytic near $\partial\Omega$. A more accurate (but rather long) name would be a {\em Schwarz function that has a meromorphic extension to $\Omega$}.

For a Borel set  $E\subset\CC$ with a compact boundary
we denote by $C^E$ the Cauchy transform of the area measure of $E$,
\begin{equation*}
 C^E(z)=\frac{1}{\pi}\int_E k_z(w)\,\dd A(w),\qquad k_z(w):=\frac{1}{z-w}.
\end{equation*}
As usual,  we understand the integral in the sense of principal value if $E$ is unbounded, e.g.
\begin{equation*}
C^{\CC}(z)=\overline z.
\end{equation*}
Indeed, if
 $R>|z|$, then
\begin{equation*}
\frac{1}{\pi}\int_{|w|\leq R}\frac{\dd A(w)}{z-w}-\overline z=
\frac{1}{\pi}\int_{|w|\leq R}\overline\partial_w\left(\frac{\overline w}{z-w}\right)\dd A(w)
= \frac{1}{2\pi\ii}\oint_{|w|=R}\frac{\overline w}{w-z}\dd w= \frac{1}{2\pi\ii}\oint_{|w|=R}\frac{R^2/w}{w-z}\dd w=0,
\end{equation*}
where we used the formula $\bar\partial k_z=-\pi\delta_z$.

The following characterization of quadrature domains is well-known, see Lemma 2.3 in \cite{AS76}. We will nonetheless outline the proof because it will give us an expression for the Schwarz function, see \eqref{SrC} below, that we will repeatedly use later.

\bigskip
\begin{lemma}\label{lem-S} $\Omega$ is a quadrature domain if and only if $\Omega$ has a Schwarz function. In this case we have
the identity
\begin{equation}\label{SrC}
S(z)=r_\Omega(z)+C^{\Omega^c}(z),\qquad z\in \clos\Omega.
\end{equation}
\end{lemma}

\begin{proof} Suppose first that $\Omega$ has a Schwarz function.
Since  $S$ is continuous up to the boundary and is finite  on $\partial \Omega$, there are only  finitely many  poles  of $S$ inside $\Omega$. Let us define $r$ as a (unique) rational function which has exactly the same poles and the same principal parts at the poles as $S$, and which satisfies $r(\infty)=0$ if $\Omega$ is bounded and
\begin{equation}\label{Sr0}
\lim_{z\to\infty} (S(z)-r(z))=0
\qquad {\rm if}\; \infty\in \Omega.
\end{equation}
 We will discuss  the {\it unbounded case}, $\infty\in \Omega$; the argument in the case of bounded domains is
 similar.

For each $z\in\Omega$ we have
\begin{equation*}
C^{\Omega^c}(z)=\frac{1}{\pi}\int_{\Omega^c} \frac{\dd A(w)}{z-w}=\frac{1}{2\pi\ii}\int_{\partial\Omega} \frac{\overline w}{w-z}\dd w=\frac{1}{2\pi\ii}\int_{\partial\Omega} \frac{S(w)-r(w)}{w-z}\dd w+\frac{1}{2\pi\ii}\int_{\partial\Omega} \frac{r(w)}{w-z}\dd w.
\end{equation*}
Here we used the fact that the boundary of $\Omega$ is rectifiable; this follows for example from Sakai's regularity theorem, which we recall in Section \ref{sec-Sakai}.
The first integral in the last expression is equal to $S(z)-r(z)$ because by \eqref{Sr0} the residue at infinity is zero.
The second integral is equal to zero -- we just apply Cauchy's theorem in each component of  the interior of $\Omega^c$. It follows that
$$S=r+C^{\Omega^c}\qquad
 {\rm inside}\;\Omega.$$ Since the Cauchy transform $C^{\Omega^c}$ is continuous in $\CC$, the identity extends to the boundary, and we have the following quadrature identity for all $f\in C_A(\Omega)$ satisfying $f(\infty)=0$:
\begin{equation*}\begin{split}
\int_\Omega f\,\dd A&=\frac{1}{2\ii} \oint_{\partial\Omega} \overline z\,f(z)\,\dd z=\frac{1}{2\ii} \oint_{\partial\Omega} S(z)\,f(z)\,\dd z
\\
&= \frac{1}{2\ii} \oint_{\partial\Omega} r(z)\,f(z)\,\dd z+\frac{1}{2\ii} \oint_{\partial\Omega} C^{\Omega^c}(z)\,f(z)\,\dd z
=\frac{1}{2\ii} \oint_{\partial\Omega} r(z)\,f(z)\,\dd z
\end{split}
\end{equation*}
because the function $C^{\Omega^c}(z) \,f(z)$ has a double zero at infinity.   It follows that $\Omega$ is a quadrature and $r$ is its quadrature function,  $$r_\Omega=r.$$

In the opposite direction, let us assume that $\Omega$ is a quadrature domain
and let us apply the quadrature identity \eqref{QI2} to the Cauchy kernels
$f=k_z$ with $z$ in the interior of $\Omega^c$.  Then
\begin{equation}\label{1ststep}
C^\Omega(z) = \frac{1}{\pi}\int_{\Omega} k_z\,\dd A=\frac{1}{2\pi\ii}\oint_{\partial\Omega} \frac{r_\Omega(w)}{z-w}\,\dd w = r_\Omega(z).
\end{equation}
By continuity of $C^\Omega$, we have
\begin{equation*}
 \forall z\in\partial\Omega,\qquad r_\Omega(z) +C^{\Omega^c}(z)=C^{\widehat\CC}(z)=\overline z,
\end{equation*}
which means that   $S := r +C^{\Omega^c}$ is the Schwarz function of $\Omega$.
\end{proof}

\subsection{Sakai's regularity theorem}\label{sec-Sakai}

Let $\Omega\subset\CC$ be an open set (not necessarily connected). A boundary point
$p\in\partial \Omega$ is called {\it regular} if there is a disc $B=B(p,\epsilon)$ such that $\Omega\cap B$ is a Jordan domain and $\partial \Omega\cap B$ is
a simple real analytic arc; otherwise $p$ is a {\it singular} boundary point.

We note two special types of singular points: $p\in\partial \Omega$ is a (conformal) {\it cusp} point if there is $B=B(p,\epsilon)$ such that $\Omega\cap B$ is a Jordan domain and every conformal map $\phi:\DD\to \Omega\cap B$ with $\phi(1)=p$ is analytic at 1 and satisfies $\phi'(1)= 0$;
 $p$ is a {\it  double} point if for some disc $B$,  $\Omega\cap B$ is a union of two disjoint Jordan domains such that $p$ is a regular boundary point for each of them.

\begin{figure}[ht]
\begin{center}
\includegraphics[width=0.6\textwidth]{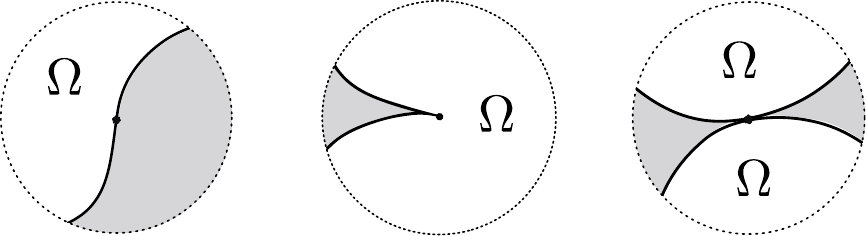}
\end{center}
\caption{\label{fig-sakai} Regular, cusp, and double points}
\end{figure}

By definition, a continuous function $S: B\cap{\rm clos}(\Omega)\to \CC$ is a {\it local} Schwarz function at $p\in\partial\Omega$ if $S$ is analytic in $B\cap\Omega$ and $S(z)=\overline z$ on
   $\partial \Omega\cap B$.

Sakai's regularity theorem  \cite{Sakai91} states that {\it if there exists a local  Schwarz function
 at a singular boundary point $p$, then $p$ is either a cusp, or a double point, or  $\Omega^c\cap B(p, \epsilon)$ is a proper subset of a real analytic  curve}.  (In the last case $p$ is called {\em degenerate}.)

In particular, if $\Omega$ has a local Schwarz function at every boundary point and if $\Omega^c$ is compact and there are no degenerate points, then the set of singular points is finite, and each singular point is a cusp or a double point.
This is the case when $\Omega$ is a quadrature domain (recall that we require $\Omega=\inte\clos\Omega$), or if $\Omega$ is the complement of an algebraic droplet.

\subsection{The complement of an algebraic droplet}\label{sec-dual}

 Let $Q(z)=|z|^2-H(z)$, $z\in\mathcal O$, be a Hele-Shaw potential, (so $H$ is harmonic on the open set ${\mathcal O}$).  Local droplets can be described in terms of the logarithmic potential of the equilibrium measure as follows.

\bigskip
\begin{lemma}\label{thm-droplet} A
 compact set $K\subset \mathcal O$ is a {\it local droplet} of $Q$ if and only if  $K$ is the support of the area measure of $K$ and  if there exists a constant $\gamma\in\RR$ such that
\begin{equation}\label{UQ}
U^K + Q = \gamma\quad \text{on~}\; K,
\qquad U^{K}(z):=\frac{1}{\pi}\int_K \log \frac{1}{|z-w|^{2}}\dd A(w).
\end{equation}
\end{lemma}
\vspace{-0.4cm}
For the proof of this fact, see, for instance, Theorem 3.3 in Ch. I of \cite{ST}.

It is clear that $U^K$ is continuously differentiable and $\partial U^K=-C^K$.
Differentiating  \eqref{UQ} and using the assumption $K={\rm supp}\, {\bf 1}_K\,\dd A$, we see that if $K$ is a local droplet, then
$$\overline z-h(z)=C^K(z) \quad {\rm on}\; K,\qquad\quad h:=\partial H.$$
In the algebraic case (i.e. when $h$ is a rational function) all poles of $h$ are in $\hat \CC\setminus K$, and therefore $h+C^K$ is a local Schwarz function at every boundary point of the open set $\widehat\CC\setminus K$. Applying Sakai's regularity theorem, we conclude that $\widehat\CC\setminus K$ has only {\it finitely many components}.

\bigskip\begin{thm}
\label{thm-HS-QD} The complement of an algebraic droplet  is a finite union of disjoint quadrature domains, and
the quadrature function of the droplet is the sum of the quadrature functions of the complementary components.
\end{thm}
\begin{proof}
Let $K$ be an algebraic droplet with quadrature function $h$, and let $\Omega_j$, $j\in\{1,\dots, N, \infty\}$,  be the complementary components, $\infty\in\Omega_\infty$. Then we have a unique representation
\begin{equation*}
h=\sum r_j,
\end{equation*}
where each $r_j$ is a rational function with poles in $\Omega_j$ and  $r_1(\infty)=\dots =r_N(\infty)=0$. We have
\begin{equation*}
\sum C^{\Omega_j}= C^{\widehat\CC}-C^K=\overline z-C^K=(h+C^K)-C^K=h=\sum r_j
\qquad\text{on }\; K,
\end{equation*}
which we can rewrite as
\begin{equation*}
C^{\Omega_\infty}-r_\infty= \sum_{j= 1}^N(r_j-C^{\Omega_j})\qquad\text{on }\; K.
\end{equation*}
The function
\begin{equation*}
F=\begin{cases}C^{\Omega_\infty}-r_\infty & {\rm on}\;  (\Omega_\infty)^c\\\sum_{j= 1}^N(r_j-C^{\Omega_j})\quad&\text{on }\; \Omega_\infty\cup K
\end{cases}
\end{equation*}
is well defined and continuous in $\widehat\CC$, zero at infinity, and analytic in $\CC\setminus \partial K$. Since $\partial K$ is  rectifiable, $F$ is entire by Morera's theorem, and therefore $F\equiv0$. It follows that
$$C^{\Omega_\infty}=r_\infty\qquad {\rm on}\;\partial \Omega_\infty,$$
and so
$$r_\infty+C^{(\Omega_\infty)^c}=C^{\widehat\CC}=\overline z\qquad {\rm on}\;\partial \Omega_\infty.$$
This means that $r_\infty+C^{\Omega_\infty^c}$ is the Schwarz function of $\Omega_\infty$, and by Lemma \ref{lem-S}  $\Omega_\infty$ is a quadrature domain and $r_{\Omega_\infty}=r_\infty$.  Applying this argument to  all bounded components $\Omega_1,\dots, \Omega_N$ (instead of $\Omega_\infty$) we conclude  that all complementary components of $K$ are quadrature domains, and that
$
r_{\Omega_j}=r_j$ for $j=1,\cdots,N$. \end{proof}
Theorem \ref{thm-HS-QD} has the following converse.
\bigskip
\begin{thm}\label{QDtodroplet} Let $K\subset\CC$ be a compact set such that the complement is a finite union of disjoint quadrature domains. Then $K$ is an algebraic droplet.
\end{thm}
\begin{proof} As in the previous proof we denote the complementary domains by $\Omega_j$, $j\in\{\infty, 1, \dots, N\}$. By assumption, $\Omega_j$'s are quadrature domains; we will write $r_j$ for the corresponding quadrature functions, $r_j=r_{\Omega_j}$, and define
$$h=\sum r_j.$$
We will construct a neighborhood $\mathcal O$ of $K$ and a harmonic function $H:{\mathcal O}\to \RR$ such that $\partial H=h$ and
\begin{equation}\label{zHU}
\forall z\in K,\qquad |z|^2-H(z)+U^K(z)=0.
\end{equation}
Since the condition $K={\rm supp}\, {\bf 1}_K\,\dd A$ is obviously satisfied (by Sakai's regularity theorem), by Lemma \ref{thm-droplet} $K$ will be a local droplet of the algebraic potential $Q(z)=|z|^2-H(z)$ and this will prove the theorem.

\noindent To construct $\mathcal O$ we can just take any open $\epsilon$-neighborhood of $K$ for sufficiently small $\epsilon$. What we need are the following properties of $\mathcal O$:
\begin{itemize}
\vspace{-0.2cm}
\item[(i)] $h$ is holomorphic in $\mathcal O$,
\vspace{-0.2cm}
\item[(ii)] each connectivity component of $\mathcal O$ contains exactly one component of $K$, and
\vspace{-0.2cm}
\item[(iii)] every loop in $\mathcal O$ (an absolutely continuous map $\gamma: S^1\to \mathcal O$) is  homotopic in $\mathcal O$ to a loop in $K$.
\end{itemize}
\vspace{-0.2cm}

The last two properties, for small $\epsilon$'s, are  immediate from the regularity theorem.
In the next paragraph we will show that (iii) implies
\begin{equation}\label{Rh0}
\Re\left(\oint_{\gamma} h(\zeta)\,\dd\zeta\right)=0\qquad {\rm for~ any~  loop} \;\gamma\;{\rm  in}\; \mathcal O.
\end{equation}
We can now construct the harmonic function $H$. Let us fix a point $z_l\in K$ in each connectivity component $\mathcal O_l$  of $\mathcal O$ and  set
$$H(z)=2\Re \left(\int_{z_\ell}^z h(\zeta)\,\dd\zeta\right)+|z_l|^2+U^K(z_l),\qquad z\in \mathcal O_l.$$
Because of \eqref{Rh0}, $H$ is a well defined real harmonic function in $\mathcal O$, and clearly $\partial H=h$.

To prove \eqref{Rh0} we first observe that
\begin{equation}\label{3star}
\forall \zeta\in K,\qquad C^K(\zeta)+h(\zeta)=\bar\zeta.\end{equation}
This is because we have  $C^{\Omega_j}=r_j$ on $\Omega_j^c$ by  the corresponding quadrature identity applied to the  Cauchy kernels as in \eqref{1ststep}, and since
 $K\subset\Omega_j^c$ for all $j$'s, we have
\begin{equation*}
 C^K(\zeta)+ \sum r_j(\zeta)=  C^K(\zeta)+\sum_{j=1}^mC^{\Omega_j}(\zeta)= \overline \zeta,\qquad(\zeta\in K).
\end{equation*}
If $\gamma$ is a loop on $K$, then
\begin{equation*}
\begin{split}2\Re\left(\oint_{\gamma} h(\zeta)\,\dd\zeta\right)&=\oint_\gamma h(\zeta)\dd \zeta+\oint_\gamma \overline{h(\zeta)}\dd \overline \zeta\\
&=\oint_\gamma \left(\overline \zeta -C^K(\zeta)\right)\dd \zeta+\oint_\gamma \left(\zeta -\overline{C^K(\zeta)}\right)\dd \overline \zeta= \oint_\gamma d\left( |\zeta|^2 + U^K(\zeta) \right)=0.
\end{split}
\end{equation*}
By the properties (i) and (iii) in the construction of $\mathcal O$, the equation extends to loops in $\mathcal O$, which proves \eqref{Rh0}.

It remains to check the identity \eqref{zHU}. The identity holds at the points $z=z_l$ by construction. On the other hand, by \eqref{3star} we have
$$\partial~\left(|z|^2-H(z)+U^K(z)\right)=\overline z-h-C^K=0\qquad {\rm on}\;K, $$
which also implies $$\overline{\partial}~\left(|z|^2-H(z)+U^K(z)\right)=0\qquad {\rm on}\;K. $$
By (ii), each component of $K$ contains one of the points $z_l$, and therefore \eqref{zHU} follows.
\end{proof}

\section{Dynamics of the Schwarz reflection}\label{sec-proof}

In this section we establish the connectivity bounds of Theorem \hyperref[thmA]{A} for quadrature domains with no singular  points on the boundary.  We will call such quadrature domains {\em non-singular}.  The argument is based on a quasiconformal modification (``surgery")
of the Schwarz reflection.

\subsection{Schwarz reflection}\label{sec-Sch-refl}

Let $\Omega$ be a {\em non-singular} quadrature domain, i.e. $\partial \Omega$ is a (finite) union of disjoint simple real-analytic curves, and let $S$ denote the Schwarz function of $\Omega$. We will study iterations of the map
$$\overline S:~\clos\Omega\to\widehat\CC,\qquad  z\mapsto \overline{S(z)}.$$
Since $\partial\Omega$ is real analytic, $\overline S$ extends to an antiholomorphic function in some neighborhood of $\clos\Omega$;
this extension is an involution  in a neighborhood of $\partial \Omega$ and  $\partial \Omega$ is the fixed set of the involution. In other words, we can think of
 $\overline S$ as an extension of the Schwarz reflection in $\partial \Omega$.

Let us denote
\begin{equation}\label{def-K}
K:=\widehat\CC\setminus\Omega.
\end{equation}
By assumption, $K$ is a finite union of disjoint closed Jordan domains.  We introduce two disjoint sets
\begin{equation*}
\overline S^{-1}\Omega=\{z\in \Omega\,|\, \overline {S(z)}\in\Omega\},
\qquad \overline S^{-1}K=\{z\in \clos\Omega\,|\, \overline {S(z)}\in K\}.
\end{equation*}
The first set is open and the second one is closed. It will be important that
since there are no singular points, the set $\overline S^{-1}\Omega$ is separated from $K$.

\bigskip
\begin{lemma}\label{lem-deg}
The maps
$$\overline S:~\overline S^{-1}\Omega~\stackrel{d}{\to}~\Omega,
\qquad \overline S:~\inte\left(\overline S^{-1}K\right)~\stackrel{d+1}\longrightarrow ~\inte K$$ are branched covering maps of  degrees $d$ and $d+1$ respectively, where
\begin{equation*}
d:=\begin{cases}
d_\Omega\quad&\text{for unbounded }\Omega,\\
d_\Omega-1 &\text{for bounded } \Omega.\end{cases}
\end{equation*}
If $\overline S$ has no critical values on $\partial\Omega$ then
$$\overline S:~\partial (\overline S^{-1}K)=\partial (\overline S^{-1}\Omega)\sqcup\partial {K}~\stackrel{d+1}{\longrightarrow}~\partial {K}$$ is a covering map.
\end{lemma}

\begin{proof}
Following Gustafsson \cite{Gu1}, we consider the Schottky double ${\mathcal M}=\Omega^{\rm double}$ of the domain $\Omega$. Recall that ${\mathcal M}$ is a union of two copies of $\clos\Omega$, which we denote $(\clos\Omega, 1)$ and $(\clos\Omega,2)$, with the identification
$$(z,1)\sim(z,2),\qquad z\in\partial \Omega,$$
along the boundary.
There is a unique complex structure on ${\mathcal M}$ consistent with the charts
$$(z,1)\mapsto z,\quad (z,2)\mapsto \overline z,\qquad (z\in\Omega),$$
and with respect to this complex structure, the map
$$(z,1)\mapsto S(z),\quad (z,2)\mapsto \overline z,\qquad (z\in\Omega)$$
extends to a meromorphic
function
$$F:{\mathcal M}\longrightarrow \widehat\CC. $$
The degree of $F$ is the number of preimages of $\infty$, which is $d_\Omega$ on the first sheet ($S$ has the same poles as $r_\Omega$ by Theorem \ref{lem-S}), and 1 or 0 on the second sheet according as $\infty$ is in $\Omega$ or not. It follows that
$$\deg F=d+1.$$
Restricting the branched cover $\overline F:{\mathcal M} \stackrel{d+1}\longrightarrow \widehat\CC $ to the preimage of $\Omega$ and disregarding the component $(\Omega,2)$ in this preimage, we obtain the branched cover $\overline S:\overline S^{-1}\Omega\stackrel{d}{\to}\Omega$. Since  $\overline F^{-1}(\inte K)\subset(\clos\Omega, 1)$, the restriction of $\overline F$ to $\overline F^{-1}(\inte K)$ gives us $\overline S:\inte(\overline S^{-1}K)\stackrel{d+1}\longrightarrow \inte K$. The last statement of the lemma is obvious.
\end{proof}

{\bf Remarks.}
\begin{itemize}
\item[(a)] {\em Algebraicity}

 With minor modification (prime ends instead of boundary points), the Schottky double construction extends to general quadrature domains (which may have singular points on the boundary).  We have two meromorphic functions $F$ and $F^\#$ on ${\mathcal M}=\Omega^\text{double}$, where  $F^\#(z,1)=z$ and $F^\#(z,2)=\overline{S(z)}$.
This implies that the boundary of any quadrature domain is a real algebraic curve of degree $\leq 2(d_\Omega+1)$, see \cite{Gu1} for details.

\item[(b)]
{\em Circular inversion of a quadrature domain is a quadrature domain}

 Let $\Omega$ be a BQD with $0\in \Omega$, or a UQD with $0\not\in \clos\Omega$. The circular inversion of $\Omega$,
\begin{equation*}
\widetilde\Omega=\{z: 1/\overline z\in\Omega\},
\end{equation*}
has the Schwarz function
$
\widetilde S(z)=1/\overline{S(1/\overline z)}$.
Therefore, by Theorem \ref{lem-S}, $\widetilde\Omega$ is a quadrature domain, and its order is given by the number of zeros of $S$ in $\Omega$ (counted with multiplicities).   By Lemma \ref{lem-deg}, this number is $d_\Omega+1$  for bounded domains and $d_\Omega$ for unbounded domains.  Thus we have a one-to-one correspondence between BQDs of order $d+1$ and
UQDs of order $d$.
\end{itemize}

\subsection{Model dynamics}

Our strategy will be to extend the map $\overline S:S^{-1}\Omega\to\Omega$ to a topological branched cover of the Riemann sphere and then apply the Douady-Hubbard straightening construction. For simplicity we first assume that
\begin{equation*}
\text{$\overline S$ has no critical value on $\partial\Omega$,}
\end{equation*}
i.e. $z\in\Omega$, $\overline{S(z)}\in\partial\Omega~ \Rightarrow ~S'(z)\neq 0$; in this case $\partial(\overline S^{-1} K)$ is a union of disjoint real analytic Jordan curves.   In Section \ref{sec-surgery} we outline a simple modification of the argument in the case when $\overline S$ has a critical value on $\partial\Omega$.
Let
\begin{equation}\label{decomp-K}
K=\bigsqcup_{\ell=1}^{\conn\Omega}K_\ell
\end{equation}
be the decomposition of $K$ into connected components. (As we mentioned, each $K_\ell$ is a closed Jordan domain.) Accordingly, we have
$$\overline S^{-1}K=\bigsqcup_{\ell=1}^{\conn\Omega}\overline S^{-1}K_\ell $$
where the sets $S^{-1}K_\ell=\{z\in \clos\Omega\,|\, \overline {S(z)}\in K_\ell\}$ are not necessarily connected.  We denote by $A_\ell$ the component of $\overline S^{-1}K_\ell$ that contains $\partial K_\ell$.
As we mentioned, this component contains an annulus such that $\partial K_\ell$ is one of the boundaries of the annulus.   Filling in the hole in the annulus, we also define
$$\hat A_\ell=A_\ell\cup K_\ell,$$
and set
$${A}:=\bigsqcup_{\ell=1}^{\conn\Omega}{A}_\ell,\qquad {\hat A}:=\bigsqcup_{\ell=1}^{\conn\Omega}{\hat A}_\ell.$$
We will modify $\overline S$ on the set $A$ by extending the covering map $\overline S:\partial\hat A\to\partial K$ to a smooth branched cover $\hat A\to K$.  We construct such an extension for each component $\hat A_\ell$ separately using the following ``model" dynamics.  The construction is illustrated in Figure \ref{fig-model}.

\bigskip
\begin{lemma}\label{lem-model} Given positive integers $\nu_1,\dots, \nu_m$, consider the rational function
$$f(z)=\epsilon^2 g(z),$$
where $\epsilon>0$ is a sufficiently small  number and
$$g(z)=z^{\nu_m}\left[1+\sum_{k=1}^{m-1}\frac 1{(z-k)^{\nu_k}}\right].$$
If we denote $$V=\{|z|<\epsilon\},\qquad U=\overline f^{-1}V=\{|f(z)|<\epsilon\},$$
then the following holds true:
\begin{itemize}
\vspace{-0.4cm}
\item[(i)] $U$ is a bounded domain of connectivity $m$, and the map
$$\overline f:~U~\stackrel{\nu}\to~ V,\qquad \nu:=\sum_{k=1}^{m}\nu_k,$$
is a branched covering map of degree $\nu$;
\vspace{-0.4cm}
\item[(ii)] the connected components of $\partial U$ are real analytic Jordan curves, and the restrictions of $\overline f:\partial U\to \partial V$ to these curves are covering maps of degrees $\nu_1,\dots, \nu_m$ respectively;
\vspace{-0.4cm}
\item[(iii)] $\clos V\subset U$,  $0\in V$ is an attracting fixed point of $\overline f$, and the  orbits of all points in $U$ are attracted to $0$.
\end{itemize}
\end{lemma}

\begin{proof}  $f$ has poles at the points $z=1, \dots, m-1, \infty$, and the multiplicities of the poles are $\nu_1,\dots, \nu_{m-1},\nu_m$ respectively. In particular, $\deg f=\nu$.  If $\epsilon$ is very small, the boundary of $U$,
$$\partial U=\left\{|g(z)|=\frac1\epsilon\right\}, $$
consists of small Jordan curves that are close to circles surrounding the points $1, \dots, m-1$, and a large ``circle" around infinity.   All the statements of the lemma are obvious.
\end{proof}

\begin{figure}[ht]
\begin{center}
\includegraphics[width=0.7\textwidth]{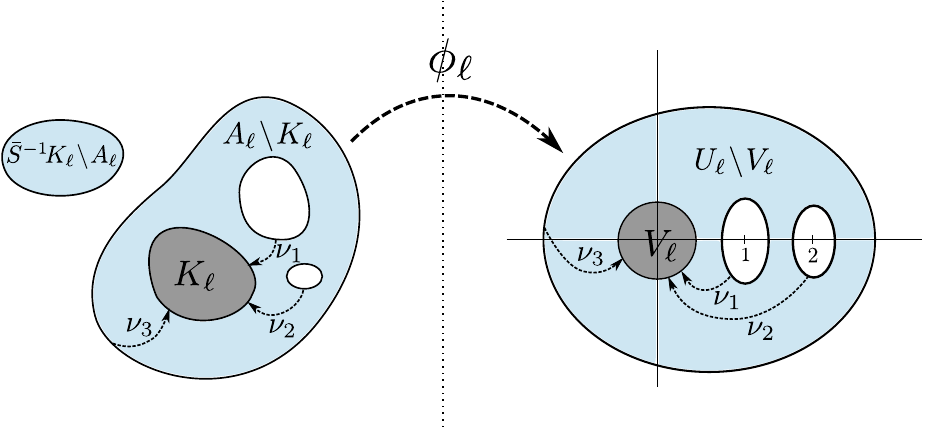}
\end{center}
\caption{\label{fig-model} Modification of the Schwarz reflection in Lemma \ref{Dtheorem}}
\end{figure}

\subsection{Douady-Hubbard construction}\label{sec-surgery}

\begin{lemma}\label{Dtheorem}
There exists a branched covering map $G:\widehat\CC\to\widehat\CC$ of degree $d$ such that
\begin{itemize}
\vspace{-0.4cm}
\item[(i)] $G=\overline S$ on $\widehat\CC\setminus \hat A$,
\vspace{-0.4cm}
\item[(ii)] $G$ is quasi-conformally equivalent to a rational map, i.e.
\begin{equation*}
G=\Phi^{-1}\circ \overline R\circ \Phi,
\end{equation*}
for some rational map $R$ and some orientation preserving quasi-conformal homeomorphism $\Phi:\widehat\CC\to\widehat\CC$,
\vspace{-0.4cm}
\item[(iii)] each component $K_\ell$ of $K$ contains a fixed point of $G$ which attracts the orbits of all points in $K_\ell$.
\end{itemize}
\end{lemma}
\begin{proof} For each component $K_\ell$  we construct the map
$$\overline f_\ell: U_\ell\to V_\ell$$
as in the previous lemma in which  we choose $m$ to be the connectivity of $\hat A_\ell$ and $\nu_1,\dots,\nu_m$ the degrees of $\overline S$ on the components of $\partial\hat A_\ell$. (Recall that $\overline S:\partial\hat A_\ell\to\partial K$ is a covering map.) Let
$$\phi_\ell:K_\ell\to \clos V_\ell$$
be (the continuous extension of) a Riemann mapping. We can lift $\phi_\ell:\partial K_\ell\to \partial V_\ell$ to a diffeomorphism $\phi_\ell:\partial \hat A_\ell\to \partial U_\ell$ so that
\begin{equation*}\begin{split}
\text{
\begin{tikzpicture}[node distance=3cm, auto]
\node (H) {$\displaystyle \partial \hat A_\ell$};
  \node (K) [node distance=2cm, below of=H] {$\partial K_\ell$};
  \node (N) [right of=H] {$\displaystyle \partial U_\ell$};
  \node (D) [node distance=2cm, below of=N] {$\partial V_\ell$};
  \draw[->,dashed] (H) to node {$\phi_\ell$} (N);
  \draw[->] (H) to node [swap] {$\overline S$} (K);
  \draw[->] (K) to node {$\phi_\ell$} (D);
   \draw[->] (N) to node {$\overline{f_\ell}$} (D);
\end{tikzpicture}}\end{split}
\end{equation*}
This is possible because we have matched the degrees of the covering maps. Finally, we extend $\phi_\ell$ to a diffeomorphism
$$\phi_\ell:\hat A_\ell\to \clos U_\ell.$$
We define the map $G$ by the formula
\begin{equation*}
G=\begin{cases} \overline S ~~&\text{on}~~\Omega\setminus  A,\\
\phi_\ell^{-1}\circ \overline{f_\ell}\circ\phi_\ell~~&\text{on}~~{\hat A}_\ell\cup K_\ell,\quad \ell=1,\cdots,\conn\Omega.
\end{cases}
\end{equation*}
By construction, $G:\widehat\CC\to \widehat\CC$ is a branched cover of degree $d$.

 Following  the proof of Douady-Hubbard straightening theorem \cite{DH}, let us show that $\overline G$ is quasi-conformally equivalent to a rational function.. We will construct an invariant infinitesimal ellipse field $\mathcal E$ of bounded   eccentricity and then  apply
the measurable Riemann mapping theorem (see e.g. \cite{Carleson-Gamelin}, Section I.7 and VI.1). The invariance means that for almost all $w\in\mathbb C$,
\begin{equation}\label{pullbackE}
Gz=w\quad\Rightarrow\quad G_* {\mathcal E}(z)={\mathcal E}(w).
\end{equation}
Since $GK\subset K$ we have
$$K\subset G^{-1}K\subset G^{-2}K\subset\dots,$$
and we have the decomposition
$$\hat \CC=(\widehat\CC\setminus G^{-\infty}K)~\sqcup~ K~\sqcup~(G^{-1}K\setminus K)~\sqcup~(G^{-2}K\setminus G^{-1} K)~\sqcup~\dots,$$
where $$G^{-\infty}K:=K\cup G^{-1}K\cup G^{-2}K\cup\dots$$
Let us set
$${\mathcal E=\mathcal E_0}~~ (\text{circle field}) \quad{\rm on}\quad (\widehat\CC\setminus G^{-\infty}K)~\sqcup~ K$$
and then define $\mathcal E$ on $G^{-\infty}K$  by recursively  applying  \eqref{pullbackE}. The resulting ellipse field has bounded   eccentricity because $G=\overline S$ outside $G^{-1}K$ (we have $GA\subset K$ so $A\subset G^{-1}K$). The field is invariant because \eqref{pullbackE} is automatic by construction if  $z\notin K$. If $z\in K$, then $G$ is conformal at $z$, and  $w=Gz\in K$.  By the measurable Riemann mapping theorem, there exists a quasi-conformal homeomorphism $\Phi:\widehat\CC\to\widehat\CC$ such that
$$\Phi^*{\mathcal E}={\mathcal E}_0.$$
The branched covering map $\overline R:=\Phi\circ G\circ \Phi^{-1}$ takes infinitesimal circles to circles, so $R$ has to be a rational function.
\end{proof}

\begin{figure}[ht]
\begin{center}
\includegraphics[width=0.7\textwidth]{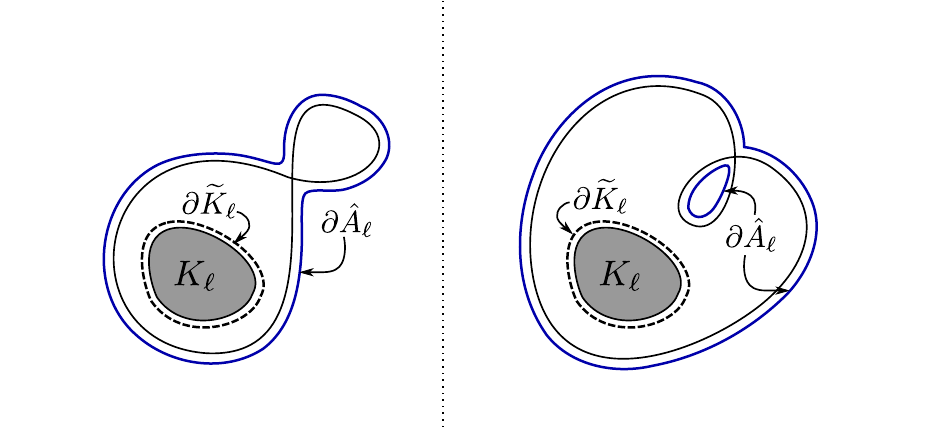}
\end{center}
\caption{\label{fig-model-crit} The case where $\bar S$ has a critical value on $\partial K$}
\end{figure}

{\bf Remark.}

In the case where $\overline S$ has a critical value on $\partial\Omega$ the statement of Lemma \ref{lem-deg} remains true if we redefine the set $\hat A$ as follows.  (We still assume that $\Omega$ is non-singular.)  Fix a sufficiently small positive number $\epsilon$, denote by $\widetilde K$ the complement of the $\epsilon$-neighborhood of $\Omega$, and set
$$\hat A=\overline S^{-1}\widetilde K\cup \widetilde K,$$
see Figure \ref{fig-model-crit}.   The restriction $\overline S:\partial \hat A\to\partial\widetilde K$ is an unbranched covering map, which we can extend to a smooth branched cover $\hat A\to\widetilde K$ using the model dynamics as in the proof of Lemma \ref{lem-model}.   The rest of the argument goes through verbatim.


\subsection{Connectivity bounds for non-singular unbounded quadrature domains}\label{proof-UQD}

We will now use Lemma \ref{Dtheorem} and elementary facts of rational dynamics to prove the connectivity bounds of Theorem \hyperref[thmA]{A} in the case of non-singular quadrature domains.  The proof, which is quite similar to Khavinson-\'Swi\c atek's argument in \cite{Kha1}, is based on Fatou count of attracting fixed points.
\bigskip
\begin{lemma}\label{lem-Fatou} Let $R$ be a rational function of degree $\geq 2$.  Then each immediate basin of attraction for $\overline R$ contains at least one critical point of $R$.
\end{lemma}
If $a$ is an attracting fixed point of $\overline R$, then its {\it immediate basin of attraction} is the component of the open set $\{z:~\overline R^n(z)\to a\}$ which contains $a$. See \cite{Carleson-Gamelin}, Section III.2 for the proof of Fatou's lemma in the case of holomorphic rational dynamics.  Exactly the same proof works for anti-holomorphic dynamics and gives Lemma \ref{lem-Fatou}.

Let $\Omega$ be a quadrature domain (bounded or unbounded) without singular points on the boundary, and let $S$ be its Schwarz function.   If $\overline S$ has no critical values on $\partial\Omega$ then as in \eqref{def-K}-\eqref{decomp-K} we denote
$$K=\widehat\CC\setminus\Omega=\bigsqcup_{\ell=1}^{\conn\Omega}K_\ell$$
and construct a $d$-cover $$G:\widehat\CC\to \widehat\CC$$ as in Lemma \ref{Dtheorem}.  (Recall that $d=d_\Omega$ in the case of unbounded quadrature domains and $d=d_\Omega-1$ for bounded quadrature domains.)
If $\overline S$ has critical values on $\partial\Omega$ then we need to proceed as explained in Remark in Section \ref{sec-surgery}.
In any case,
$G$ is quasi-conformally equivalent to an anti-analytic rational map, $G$ has $2d-2$ critical points, and each component $K_\ell$ contains an attracting fixed point of $G$.  Applying Lemma \ref{lem-Fatou} we find that
\begin{equation}\label{dkn}
\conn\Omega \leq  2d-2=\begin{cases} 2d_\Omega-2 &\text{for unbounded $\Omega$},
\\ 2 d_\Omega- 4 &\text{for bounded $\Omega$}.\end{cases}
\end{equation}
It is clear that we get a better bound if there are critical points of multiplicity $>1$, or more generally if there are several critical points with the same asymptotic behavior. More specifically, let us call critical points $c_1,\dots ,c_m$ {\em equivalent} if
$$G(c_1)=G(c_2)=\dots= G(c_m).$$
If a fixed point attracts the orbit of $c_1$, then it attracts the orbits of all equivalent points as well, so the Fatou count gives us
$$\conn\Omega \leq  2d-2-\left(\sum_1^m {\rm mult}~ c_j-1\right)=2d-1-\sum_1^m {\rm mult}~ c_j.$$
Furthermore, if we somehow know that the orbit of $c_1$ is not attracted to any of the fixed points in $K$, then we get
\begin{equation*}
\conn\Omega \leq  2d-2-\sum_1^m {\rm mult}~ c_j.
\end{equation*}
Let us assume now that $\Omega$ is an {\it unbounded} domain, and let $z_1,\cdots, z_n$ be the poles of $S$ of order $\mu_1,\cdots,\mu_n$.   Since $\overline Sz_j=\infty$ we have $z_j\not\in A$ and $G=\overline S$ near the poles. The poles with $\mu_j>1$ are critical points $c_j=z_j$ of $G$ of multiplicities $\mu_j-1$. These critical points are equivalent  and
$$\sum {\rm mult}~ c_j=\sum_1^n(\mu_j-1)=\sum_1^n\mu_j-n=d-n.$$
It follows that
$$\conn\Omega \leq2d-1-(d-n)=d+n-1,$$
and together with \eqref{dkn} this proves the inequality \eqref{eq-thma1}.

Furthermore, if $\Omega$ has a node at $\infty$, then $\infty$ is a fixed point of $G$ and so the orbits of the poles are not attracted to any fixed point in $K$.  Therefore we can apply \eqref{dkn} and we get \eqref{eq-thma2},
\begin{equation*}
\conn\Omega\leq d+n-2.
\end{equation*}

\subsection{Connectivity bounds for non-singular bounded quadrature domains}\label{proof-BQD}
Let $\Omega$ be a non-singular {\it bounded} quadrature domain of order $d_\Omega\geq 3$ (so $d=d_\Omega-1\geq 2$, and we can apply Lemma \ref{lem-Fatou}), and let $S$ be its Schwarz function.
For notational simplicity we assume that $\overline S$ has no critical value on $\partial\Omega$, see Remark in Section \ref{sec-surgery}.  Denote by $K_\infty$ the unbounded component of $K=\Omega^c$ and consider the branched covering map
\begin{equation}\label{451}
\overline S:~\overline S^{-1}K_\infty\equiv A_\infty\sqcup B_1\sqcup\dots\sqcup B_m~\stackrel{d_\infty~\sqcup~ d_1~\sqcup~\dots~\sqcup~ d_m}{\xrightarrow{\hspace*{3cm}}}~K_\infty
\end{equation}
where $A_\infty$ and $B_j$ are connectivity components, $\partial K_\infty\subset A_\infty$, and
$$d_\infty+ d_1+\dots+ d_m=d_\Omega,$$
see Lemma \ref{lem-deg}.  We will first prove the inequality
\begin{equation}\label{452}
\conn\Omega\le d_\Omega+n_\Omega-2.
\end{equation}
We can assume that $\Omega$ is not simply connected (otherwise there is nothing to prove) so $d_\infty\geq 2$ (recall that $K_\infty$ is a closed Jordan domain).  Let $G$ be the map constructed in Lemma \ref{Dtheorem}.  We have
\begin{equation*} G:~G^{-1}K_\infty=\hat A_\infty\sqcup B_1\sqcup\dots\sqcup B_m~\stackrel{(d_\infty-1)~\sqcup~ d_1~\sqcup~\dots~\sqcup~ d_m}{\xrightarrow{\hspace*{3cm}}}~K_\infty,
\end{equation*}
because $G:\hat A_\infty\to K_\infty$ is a $(d_\infty-1)$-cover and $G=\overline S$ on $\bigsqcup B_j$.  By the Riemann-Hurwitz count, $G$ has {\em at least}
\begin{equation}\label{454}
(d_\infty-2)+ (d_1-1)+\dots+ (d_m-1)=d_\Omega-m-2
\end{equation}
critical points (counted with multiplicities) in $G^{-1} K_\infty$.  The $G$-orbits of all these critical points land in $K_\infty$ so {\em at most}
\begin{equation}\nonumber
(2d_\Omega-4)-(d_\Omega-m-2)=d_\Omega+m-2
\end{equation}
critical points of $G$ land in finite components of $K$.  By Lemma \ref{lem-Fatou} we have
\begin{equation}\label{455}
\conn\Omega\leq d_{\Omega}+m-1.
\end{equation}
Returning to \eqref{451} we estimate $m$ in terms of the number of {\em distinct} poles of $\overline S$ (or, equivalently, $r_\Omega$):
\begin{equation}\label{mnnn}
m+1\le n_\infty+ n_1+\dots+ n_m=n_\Omega
\end{equation}
where $n_j\geq 1$ (resp. $n_\infty\geq 1$) are the number of distinct nodes of $S$ in $B_j$ (resp. $A_\infty$). Together with \eqref{455} this gives \eqref{452}.  Applying \eqref{dkn} we get \eqref{eq-thma3}.

We would get a better estimate if there were at least two distinct nodes in $A_\infty$.

Let us now show that if $S$ has no poles of order $\geq 3$ then
\begin{equation}\label{456}
\conn\Omega\le d_\Omega+n_\Omega-3.
\end{equation}
Indeed, if the poles in $A_\infty$ are at most double, and if $d_\infty\geq 3$ then $S$ has two distinct poles in $A_\infty$ and we are done by the previous remark.   Let us therefore assume $d_\infty=2$ which means that $G$ has at least $d_\Omega-m-2$ critical points (counted with multiplicities) in $\bigcup B_j$, see \eqref{454}.  At the same time $G$ has an attracting fixed point in $K_\infty$, and by Lemma \ref{lem-Fatou} there is at leat one critical point of $G$ in the set
$$\Bigg[\bigcup_{n\ge0} G^{-n}K_\infty\Bigg]
~\setminus~\Bigg[\bigcup_{j=1}^m B_j\Bigg]$$
because the {\em immediate} basin of attraction does not intersect the sets $B_j$. Altogether we get at least $d_\Omega-m-1$ critical points landing in $K_\infty$ and at most $d_\Omega+m-3$ critical points landing in finite components of $K$.  The estimate  \eqref{456} now follows from \eqref{mnnn}, and combining  \eqref{456} and \eqref{dkn} we get \eqref{eq-thma4}.

{\bf Remark.}

There is a short way to derive the estimate $\conn\Omega\leq d_\Omega+n_\Omega-2$ from the corresponding bound \eqref{eq-thma1} for UQDs.
Inscribe $\Omega$ in a round disc $B$ centered at the origin so that there are at least two common boundary points, $\#(\partial\Omega\cap\partial B)\ge2$. Applying a Hele-Shaw perturbation, see the next section, we get an unbounded quadrature domain $\Omega'$ with $d_{\Omega'}=d_\Omega$ and $n_{\Omega'}=n_\Omega$  but with $\conn(\Omega')\ge \conn(\Omega)+1$. By \eqref{eq-thma1} we have
$$\conn\Omega'\le d_{\Omega'}+n_{\Omega'}-1,$$
which implies \eqref{452}.

On the other hand, it is not quite clear how to derive the (sharp) stronger estimate \eqref{eq-thma4} using this method in the case when all nodes of the bounded quadrature domain are at most double.

\section{Singular quadrature domains and Hele-Shaw flow}

In the last section of the paper we complete the proof of Theorem \hyperref[thmA]{A}  and also prove two other related statements (Theorems \ref{thm-theyare} and \ref{thm-B} from Section 2).  We start with a review of some properties of Hele-Shaw chains of algebraic droplets. The theory of Hele-Shaw chains will allow us to extend the connectivity bounds to quadrature domains with singular points on the boundary.

\subsection{Hele-Shaw chains with source at infinity}\label{sec-HSdetail}

Let $K$ be a local droplet for some algebraic Hele-Shaw potential $Q$, and let
$$A(K)=\pi t_0.$$
The (backward) Hele-Shaw chain of $K$ with source at infinity is the family of compact sets
\begin{equation}\label{eq:KtSt}
K_t=S_t[Q_K], \qquad 0<t\le t_0.\end{equation}
Here $Q_K$ is the localization of $Q$ to $K$ and $S_t[Q_K]$ is the support of the corresponding equilibrium measure, see Section \ref{sec-log}. Clearly,
$K_{t_0}=K$, and we can also define $K_0=\emptyset$. Below we list some properties of the chain $\{K_t\}$; see \cite{HM2011} for further information.
\begin{enumerate}
\item[(a)]\label{itema} The sets $K_t$ are algebraic droplets. They are  local droplets of $Q$, so the quadrature function is the same for all $t$.   By Sakai's regularity theorem, all boundary points of each droplet $K_t$ are regular except for a finite number of cusps and double points.
\item[(b)]\label{itemb} For every continuous function $f:\C\to\R$ we have
\begin{equation}\label{eq:HS} \frac1\pi \frac d{dt}\int_{K_t} f~dA= \int f~d\omega_t^\infty,\end{equation}
where $\omega_t^\infty$ is the harmonic measure of $K_t$ evaluated at infinity. The derivative in \eqref{eq:HS} is two-sided for $t\in (0,t_0)$ and one-sided for $t=t_0$.
The chain $\{K_t\}$ is a unique solution of \eqref{eq:HS} satisfying $K_{t_0}=K$. In particular the chain of each droplet $K_t$ is a sub-chain of the chain of $K$.
\item[(c)]\label{itemc} The chain $\{K_t\}$ is monotone increasing, and $A(K_t)=\pi t.$ In fact we have the following {\it strong monotonicity} property:
$$t<t_0\quad \Rightarrow\quad {\mathcal P}(K_t)\subset {\rm int}~{\mathcal P}(K_{t_0}),$$
where ${\mathcal P}(\cdot)$ is the notation for the polynomial convex hull (the complement of the unbounded component of the complement).
\item[(d)]\label{itemd} Hele-Shaw chains are left-continuous, e.g.
$$K_{t_0}=\clos\bigcup_{t<t_0}K_t.$$
\item[(e)]\label{iteme} The droplets $K_t$ can be described in terms of the following {\em obstacle problem}:
\begin{equation}\label{Obst}
V_t(z)=\sup \Big\{v(z):~ v\in {\rm Sub}(\CC),~v\leq Q {\rm ~on~} K, ~ v(z)\le t\log|z^2|+O(1) {\rm ~as~} z\to\infty\Big\}.
\end{equation}
Denote by $K^*_t$ the coincidence sets:
$$K^*_t=\{Q=V_t\},\qquad 0\le t\le t_0.$$
(For example, $K_0^*$ is the global minimum set of $Q:K\to \R$.)
We have
\begin{itemize}
\item $K_t\subset K_t^*$, and $K_t$ is the support of the area measure restricted to $K_t^*$;
\item $\#( K_t^*\setminus K_t)<\infty$;
\item if $0\leq t<t_0$, then $ {\mathcal P}(K_t^*)=\bigcap_{\epsilon>0}{\mathcal P}(K_{t+\epsilon})$. 
\end{itemize}
\item[(f)]\label{itemf}
 The chain $\{K_t\}$ is discontinuous in the Hausdorff metric at the set of times $t$ such that ${\mathcal P}(K_t^*)\ne{\mathcal P}(K_t)$. As $t$ increases, new components of the droplet appear at those times. Some components could merge; this happens at another set of times when the droplet has douple points on the boundary.
 It should be true that in the algebraic situation the set of such ``singular'' times is finite but we were unable to find a proper reference.  The following statement ({\em Sakai's laminarity theorem}, see \cite{Sakai-book}) will be sufficient for our purposes.  The hydrodynamical term ``laminarity'' refers to the absence of topological changes.

{\em If $K_{t_0}$ has no double points, then there exists $\epsilon>0$ such that for all $t\in (t_0-\epsilon, t_0)$, the outer boundary of the droplet $K_t$ has no singular points and ${\mathcal P}(K_t)={\mathcal P}(K^*_t)$.}

In other words, the backward Hele-Shaw equation has a (unique) classical solution on $(t_0-\epsilon,t_0)$.
\end{enumerate}

\subsection{Hele-Shaw chains with a finite source}

For perturbations of bounded quadrature domains we will need Hele-Shaw dynamics with a source at a finite point.  Let $K$ be an algebraic droplet of area $\pi t_0$ for a Hele-Shaw potential $Q$, and suppose $0\notin K$. The (backward) Hele-Shaw chain of $K$ with source at zero is the family of compact sets
$$K_t=S_t\left[Q_K+(t_0-t)\log\frac1{|z^2|}\right], \qquad 0<t\le t_0.$$
Similarly to \eqref{eq:KtSt}, this is a unique solution of the (generalized) Hele-Shaw equation:
\begin{equation*}
\frac{1}{\pi} \frac{\dd}{\dd t} 1_{K_t} = \omega^0_t,\qquad K_{t_0}=K,
\end{equation*}
where $\omega^0_t$ is the harmonic measure of $K_t$ evaluated at the origin.

Note that unlike the case when the source is at infinity, the potential and the quadrature function of $K_t$ are now changing over time $t$.  However, the dependence on time is very simple, and all the facts listed in items (a)-(f) above extend (with obvious modifications) to the finite source case.

\subsection{Proof of Theorem \texorpdfstring{\hyperref[thmA]{A}}{thm A}}\label{sec-proofA}

We established the connectivity bounds of Theorem \hyperref[thmA]{A} for {\em non-singular} quadrature domains in Sections \ref{proof-UQD} and \ref{proof-BQD}.  To extend these bounds to quadrature domains with singular points on the boundary we will use a perturbative argument which is based on Hele-Shaw dynamics.

We start with the following a priori bound: $\exists C(d)<\infty$ such that if $\Omega$ is quadrature domain (which may be singular) of order $d$, then
\begin{equation}\label{eq:conn}\conn\Omega\leq C(d).
\end{equation}
Indeed, if $c=\conn\Omega$, then the Riemann surface ${\mathcal M}=\Omega^\text{double}$ (see Section \ref{sec-Sch-refl}) has genus
$$ g({\mathcal M})=c-1.$$
On the other hand, ${\mathcal M}$ is an algebraic curve of degree
$$D\leq 2(d+1),$$
see Remark in Section \ref{sec-Sch-refl}, so we have (see, e.g., \cite{Griffiths-Harris})
$$ g({\mathcal M})\le \frac{(D-1)(D-2)}2,$$
which gives us \eqref{eq:conn} with
$$ C(d) = 2d^2+d+1.$$
This estimate  is of course a special case of the more general Harnack curve theorem.

Let us now justify the connectivity bounds in the case of general {\em unbounded} quadrature domain $\Omega$.  Consider for instance the inequality \eqref{eq-thma1} in Theorem \hyperref[thmA]{A}:
$$\conn\Omega\leq C:=\min\{d+n-1,2d-2\},$$
where $d$ is the order of $\Omega$ and $n$ is the number of distinct nodes.  By \eqref{eq:conn} we can assume that $\Omega$ has the {\em maximal} connectivity among UQDs with given values of $d$ and $n$.   Denote $K=\Omega^c$ and consider the backward Hele-Shaw chain $\{K_t\}$ of $K\equiv K_{t_0}$ with source at infinity.  We claim that $K$ has no double points and therefore by Sakai's laminarity theorem (item (f) in Section \ref{sec-HSdetail}),
$$\Omega_t:= K_t^c$$
is a non-singular UQD of the same connectivity as $\Omega$ for all $t$ sufficiently close to $t_0$.  Since $\Omega_t$ also has the same order and the same nodes as $\Omega$, the estimate $\conn\Omega_t\leq C$ (established for non-singular domains in Section 3) extends to $t=t_0$.

To see that there are no double points we argue as follows.
If there are double points then the number of components of $\inte K$ is strictly greater than the connectivity of $\Omega$.  (Here we use the fact that $\Omega$ is connected so the components of $K$ are simply connected.)  By the left continuity of Hele-Shaw chains (item (d) in Section \ref{sec-HSdetail}), each scomponent of $\inte K$ intersects the droplets $K_t$ for $t$ sufficiently close to $t_0$.  On the other hand, by strong monotonicity (item (c) in Section \ref{sec-HSdetail}) we have
 $$K_t\subset {\rm int} {\mathcal P}(K)=\inte \Omega^c= \inte K,$$
and it follows that $K_t$ has at least as many components as $\inte K$ and therefore $\Omega_t$ has a higher connectivity than $\Omega$, which contradicts our assumption that $\Omega$ has the maximal connectivity.

The proof of the inequality \eqref{eq-thma2} for UQDs with a node at infinity is exactly the same because the quadrature function does not change under Hele-Shaw flow, so all quadrature domains $\Omega_t$ have a node at infinity.

We need to slightly modify the argument in the case of {\em bounded} quadrature domains (inequalitieis \eqref{eq-thma3} and \eqref{eq-thma4} in Theorem \ref{thm-A2}).
If $\Omega$ is a quadrature domain of maximal connectivity, then we define
$$K=\Omega^c\cap\clos B(0,R),$$
where $R$ is  large enough so that $\partial K\subset B(0,R)$. By Theorem \ref{QDtodroplet}, $K$ is an algebraic droplet with the same quadrature function as $\Omega$. Choosing one of the nodes of $\Omega$ as a finite source, we consider the corresponding backward Hele-Shaw chain $\{K_t\}$ of $K$. The quadrature function is changing but the poles and their multiplicities remain the same. We use the maximality of $\Omega$ to show  that $K$ has no double points and then we use the laminarity theorem to conclude that the droplets $K_t$ are non-singular if $t$ is sufficiently close to $t_0$. The bounded component $\Omega_t$ of $(K_t)^c$ is a non-singular quadrature domain, which has the same connectivity as $\Omega$.  It follows that the inequalities \eqref{eq-thma3} and \eqref{eq-thma4} extend to arbitrary bounded quadrature domains.

\subsection{Proof of Theorem \ref{thm-B} (topology of algebraic droplets)}\label{sec-proof-thmC}

We will first derive the bound \eqref{ovals} on the number of ovals of an algebraic droplet, and then we will show that this inequality is precisely a necessary and sufficient condition for the possible topology of a droplet of a given degree.

(i)  For a BQD of  order $d\ge 3$ and connectivity $c$, according to Theorem \hyperref[thmA]{A}  we have $c\le 2d-4$, i.e. $d\ge 2+c/2$, so
\begin{equation}\label{lowestd1}
c=1\;\Rightarrow\;d\ge1, \qquad c\ge2\;\Rightarrow\;d\ge3+\left[\frac {c-1}2\right].
\end{equation}
 For an UQD of  order $d\ge 2$ (i.e. $d+1\ge3$) we have $c\le 2(d+1)-4$, and
\begin{equation}\label{lowestd2}
c=1\;\Rightarrow\;d+1\ge1, \qquad c\ge2\;\Rightarrow\;d+1\ge3+\left[\frac {c-1}2\right].
\end{equation}
Let $K$ be a non-singular algebraic droplet of degree $d$.  By Theorem \ref{thm-HS-QD}, $K^c$  is a disjoint union of a single UQD and some BQDs. Let the orders of these quadrature domains be $d_\infty$ and $d_j$'s respectively, write
$$d:=d_\infty+\sum d_j,$$
and let $c_\infty$ and $c_j$'s denote the connectivities of the quadrature domains.
We have
\begin{align*}
d+1&=(d_\infty+1)+\sum d_j\geq q_1 + \sum_{k=2}^\infty\left( 3+\left\lfloor\frac {k-1}2\right\rfloor\right) q_k\\  &= q_1 + 3 q_2 + 4 (q_3 +q_4)+ 5(q_5+ q_6) + \cdots
\end{align*}
where $q_k$ is the number of quadrature domains of connectivity $k$, so $q=\sum q_k$. It follows that
\begin{align*}
d+1&\ge  -2q_1 +2q+ (q_1+q_2) + 2(q_3 +q_4)+ 3(q_5+ q_6) + \cdots\\
&=-2q_1 +2q+\frac12(q_1+2q_2+3q_3+\dots)+\frac12(q_1+q_3+q_5+\dots)\\
&=-2q_1 +2q+\frac12\#(\text{ovals})+\frac12 q_{\rm odd}.
\end{align*}
(Each oval is a boundary component of exactly one quadrature domain, so  $q_1+2q_2+3q_3+\dots$ is  the number of ovals.)
This proves the inequality
$$ \#(\text{ovals}) + q_\text{odd} + 4(q-q_1)\leq 2d+2.$$
Note that we have the case of equality if the degrees of all quadrature domains are equal to their lower bounds in terms of connectivities given in \eqref{lowestd1} and \eqref{lowestd2}.

(ii) To prove the second part of Theorem \ref{thm-B} we start with the following observation.

\bigskip
\begin{lemma}\label{lemma-maximal}
Given $c\geq 2$  there exists a non-singular unbounded quadrature domain $\Omega$ of order $2+\left[\frac {c-1}2\right]$ such that  ${\rm conn}(\Omega)=c$, and similarly there exists a non-singular bounded quadrature domain of order $3+\left[\frac {c-1}2\right]$ and  connectivity $c$.
\end{lemma}

\begin{proof} This lemma is essentially a statement about the sharpness of connectivity bounds of Theorem \hyperref[thmA]{A} in several special cases.  As we mentioned all bounds are sharp, see Theorems \ref{thm-B1} and \ref{thm-B2}, and while we prove these Theorems in full generality in a separate paper \cite{paper-sharp}, the special cases under consideration are quite elementary and could be derived here without any additional tools.

\begin{itemize}
\item
If $c$ is even, $c=2k$, then we claim there exists a non-singular unbounded quadrature domain $\Omega$ of connectivity $c$ with $k+1$ finite simple nodes (so $d_\Omega=k+1=2+[\frac{c-1}{2}]$).  This is the case $n=d$ in the first statement of Theorem \ref{thm-B1}.  We will use the construction described in Section \ref{sec-HSflow}.

Let us inscribe $k+1$ disjoint open discs in a big closed disc so that the interior of the complement has $2k$ components.  The case $k=1$ is shown in the first picture in Figure \ref{fig-proof}, for $k\geq 2$ we can use induction (``Apollonian packing'').  The complement (the shaded region in the picture) is an algebraic droplet; its quadrature function has $k+1$ simple finite poles, see \eqref{HS-QD} and \eqref{extcircle}.  Let $\{K_t\}_{0<t\leq t_0}$ be the Hele-Shaw chain (with source at infinity) of this droplet.  By strong monotonicity and left continuity of the Hele-Shaw flow, $K_t$ has at least $2k$ components if $t$ is close to $t_0$.  In fact, we have exactly $2k$ components, and there are no double points because $2k$ is the maximal connectivity of an unbounded quadrature domain with $d=n=k+1$, see \eqref{eq-thma1} in Theorem \ref{thm-A1} and the argument in Section \ref{sec-proofA}.  Thus we can apply Sakai's laminarity theorem to obtain a non-singular quadrature domain with the same quadrature function.

\item
If $c$ is odd, $c=2k+1$, then there is a non-singular unbounded quadrature domain of connectivity $c$ which has one finite double node, $k-1$ finite simple nodes, and a simple node at infinity. Note that $d_\Omega=k+2=2+\left[\frac {c-1}2\right]$.  This is the case $n=d-1$ in the second statement of Theorem \ref{thm-B1}.

The case $k=1$ is illustrated in the second picture in Figure \ref{fig-proof}: we inscribe a cardioid in an ellipse so that the interior of the complement $K$ (the shaded region) has 3 components. If $k\geq 2$ we also inscribe $k-1$ disjoint open disks in $K$ so that the interior of the complement has $2k+1$ components. It is easy to justify the existence of such ``Apollonian'' packing using convexity considerations.  Applying Hele-Shaw flow with source at infinity, we get an unbounded quadrature domain of connectivity $2k+1$.  It is important that according to \eqref{eq-thma2} in Theorem \ref{thm-A1}, $2k+1$ is the maximal connectivity for unbounded quadrature domains with a node at infinity such that $d=k+2$ and $n=k+1$.  There are no double points by the argument in Section \ref{sec-proofA}, and we get a non-singular quadrature domain with the same quadrature function by the Hele-Shaw flow.

\item If $c=2k$, then there exists a bounded quadrature function $\Omega$ of connectivity $c$ with $k+2$ simple nodes. The proof is exactly the same as in the case of unbounded domains except that we use the Hele-Shaw flow with a {\em finite} source.  The setup for Apollonian packing  ($k=1$) is shown in the 3rd picture in Figure \ref{fig-proof}.  The maximality of the connectivity follows from \eqref{eq-thma4} in Theorem \ref{thm-A2}.

\item If $c=2k+1$, then there exists a bounded quadrature function $\Omega$ of connectivity $c$ with two double points and $k-1$ simple nodes.  The case $k=1$ is shown in the 4th picture in Figure \ref{fig-proof}.
\end{itemize}
\vspace{-0.92cm}
\end{proof}

\begin{figure}[ht]
\includegraphics[width=\textwidth]{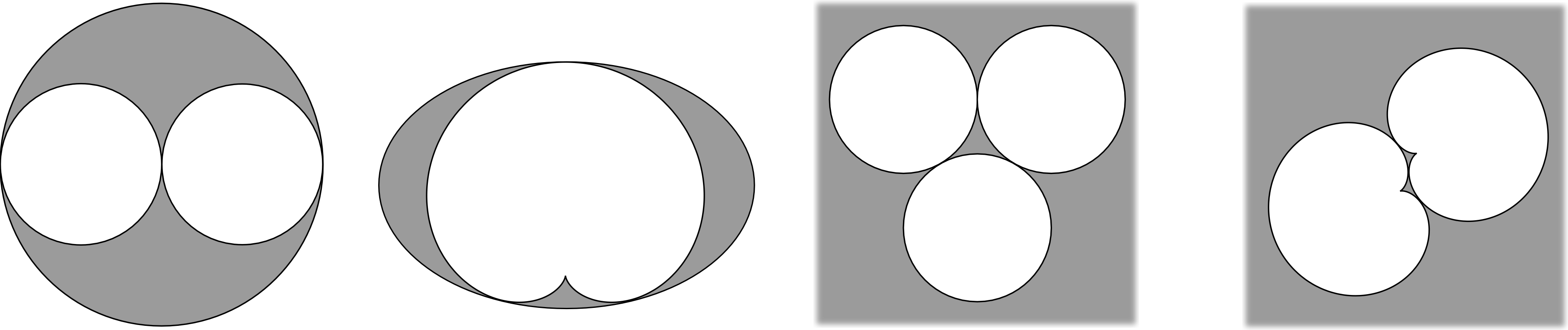}\caption{\label{fig-proof} Initial setups for Apollonian packing in Lemma \ref{lemma-maximal}}
\end{figure}

Let us now finish the proof of the theorem.
Given $d$ and some configuration of ovals satisfying \eqref{ovals}, we want to construct a non-singular algebraic droplet $K$ of degree $d$ such that $\partial K$ is topologically equivalent to the given configuration.  Let $\widetilde K$ denote the compact set bounded by the given ovals.  We can describe its topology by an (oriented) rooted tree as follows.  The vertices of the tree are complementary domains $U$ of $\widetilde K$ with the unbounded domain being the root.  The edges are the pairs $[U_1,U_2]$ such that $U_2$ sits inside some bounded component of $U_1^c$ and there is a curve in $\widetilde K$ connecting $\partial U_1$ and $\partial U_2$.  Using the induction with respect to the graph distance from the root, we can construct a family $\{\Omega\}$ of disjoint non-singular quadrature domains 
such that the corresponding oval configuration is equivalent to $\partial \tilde K$, the orders of non-simply connected $\Omega$'s are related to their connectivities as in Lemma \ref{lemma-maximal} and the simply connected $\Omega$'s are round disks.

\begin{figure}\begin{center}
\includegraphics[width=0.6\textwidth]{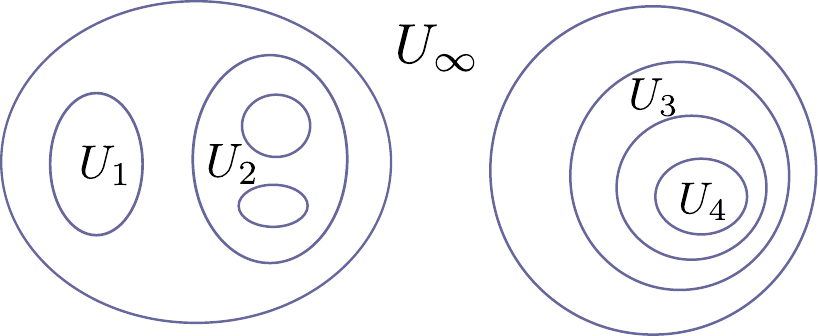}
\caption{\label{fig:induction} Example in the proof of Theorem \ref{thm-B}}
\end{center}
\end{figure}
We will use the following example to explain the construction. Consider the configuration of 9 ovals in Figure \ref{fig:induction}.  Let $U_\infty$ denote the unbounded component of $\widehat\CC\setminus \widetilde K$, and let $U_1, \dots, U_4$ be the bounded components as shown in the picture.  We first find UQD $\Omega_\infty$ such that $\conn\Omega_\infty=\conn U_\infty=2$ and $\text{order}\,\Omega_\infty=2$  (as in Lemma \ref{lemma-maximal}). At the next step we deal with components $U_1, U_2, U_3$ which are at distance 1 from the root. Inside one of the component of $\Omega_\infty^c$ we find disjoint BQDs $\Omega_1$ and $\Omega_2$ such that  $\Omega_1$ is a round disc, $\Omega_2$ is triply connected and has order 4 (again as in Lemma \ref{lemma-maximal}). We  also choose a doubly connected QD $\Omega_3$ of order 3 inside the second component of $\Omega_\infty^c$. Finally, we deal with $U_4$ which is at distance 2 from the root.  Inside the bounded component of $\Omega_3^c$ we choose a disk $\Omega_4$.

As was mentioned earlier, see the end of part (i) of the proof,  we have
\begin{equation*}
\#(\text{ovals}) + q_\text{odd}+ 4(q-q_1)  = 2d'-2,
\end{equation*}
where $d'\leq d$ is the degree of the droplet $K'=\CC\setminus\bigcup\Omega$.  If $d'=d$ then we set $K=K'$ but if $d'<d$ then we define $K$ as a droplet in the Hele-Shaw chain of $K'$ with $d'-d$ new finite sources.  This completes the proof of Theorem \ref{thm-B}.


\subsection{Proof of Theorem \ref{thm-theyare} (inverse moment problem)}\label{sec-proof-unique}

Let $h$ be a quadratic polynomial, and let $K\neq \widetilde K$ be two algebraic droplets with quadrature function $h$ and area $\pi T$.  By Corollary \ref{cor1} and  Theorem \ref{thm-HS-QD}, the droplets are connected; in particular, $K$ and $\widetilde K$ are local droplets of the Hele-Shaw potential
$$ Q(z)=|z|^2 + \Re\int_0^z f(\zeta)\,\dd \zeta.$$
We will consider the backward Hele-Shaw chain $K_t$ and $\widetilde K_t$, ($0<t\leq T$), as well as the coincidence sets $K^*_t$ and $\widetilde K^*_t$,  ($0<t\leq T$), see Section \ref{sec-HSdetail}.  The connectedness of the droplets implies
\begin{equation*}
K_t=K^*_t,\quad \widetilde K_t=\widetilde K^*_t, \quad  (0<t\leq T).
\end{equation*}
To see this, we use the facts mentioned in Section \ref{sec-HSdetail}, item (e).
The equalities are obvious for $t=T$. If $t<T$ and $K_t\neq K^*_t$, then $K^*_t$ is $K_t$ plus several isolated points, so we can find disjoint neighborhoods $U'$ of $K_t$ and $U''$ of $K^*_t\setminus K_t$.  Since $U:=U'\sqcup U''$ is a neighborhood of $K^*_t$, we have $K_{t+\epsilon}\subset U$ for all small $\epsilon>0$.  This is impossible because both $U'$ and $U''$ contain points of $K_{t+\epsilon}$ but $K_{t+\epsilon}$ is connected.

We next observe that
\begin{equation}\label{K=K}
	K^*_0=\widetilde K^*_0.
\end{equation}
This is because both sets are non-empty and all points in $K^*_0\cup \widetilde K^*_0$ are non-repelling fixed points of the map $z\mapsto \overline{h(z)}$.  Indeed, if $z_0\in K^*_0\cup \widetilde K^*_0$ then $z_0$ is a local minimum of $Q$ so
\begin{equation*}
	\partial Q(z_0)=\overline{z_0}-h(z_0)=0,
\end{equation*}
and
\begin{equation*}
	1-|h'(z_0)|^2=(\overline\partial\partial Q)(z_0) - |\partial^2 Q(z_0)|^2=\frac{1}{4}
	\begin{vmatrix}Q_{xx} & Q_{xy} \\ Q_{xy} & Q_{yy}
	\end{vmatrix}_{z=z_0}\geq 0
\end{equation*}
There could be only one such non-repelling fixed point for a given quadratic polynomial $h$, see \cite{Kha1} for the result concerning general polynomials, so we have \eqref{K=K}.

Let us now define
\begin{equation*}
\tau=\sup\{t:K^*_t=\widetilde K^*_t,~t\leq T\}.
\end{equation*}
The supremum is in fact the maximum, $K^*_\tau=\widetilde K^*_\tau$; for $\tau>0$ this follows from the left continuity of the Hele-Shaw chains, see item (d) in Section \ref{sec-HSdetail}.  To prove the theorem we need to show $\tau=T$.  We will use the following simple fact concerning local droplets, see Lemma 5.2 in \cite{HM2011}:
\begin{equation}\label{fromHM}
\text{\em if $\Sigma_1\subset\Sigma_2$ are two compact sets of positive area and if $S_t[Q_{\Sigma_2}]\subset \Sigma_1$, then $S_t[Q_{\Sigma_1}]=S_t[Q_{\Sigma_2}]$.}
\end{equation}
Suppose $\tau<T$. We choose $\Sigma_1=K_{\tau+\epsilon}$ for sufficiently small $\epsilon>0$ and $\Sigma_2=\widetilde K$.  By item (c) in Section \ref{sec-HSdetail} we have $K_{\tau +\epsilon}\subset \widetilde K$ for a small $\epsilon$ because
\begin{equation*}
K^*_\tau=\widetilde K^*_\tau\subset \inte \widetilde K.
\end{equation*}
By the same argument, if $t>\tau$ is sufficiently close to $\tau$, then we also have
\begin{equation*}
S_t[Q_{\widetilde K}]\equiv \widetilde K_t\subset K_{\tau+\epsilon}
\end{equation*}
because $\widetilde K^*_\tau=K^*_\tau\subset K_{\tau+\epsilon}$.   Applying \eqref{fromHM} we get
\begin{equation*}
K_t=S_t[Q_{K_{\tau+\epsilon}}]=S_t[Q_{\widetilde K}]=\widetilde K_t.
\end{equation*}
It follows that $\tau$ is not a supremum, a contradiction.

\bibliographystyle{alpha}

\end{document}